\documentclass[11pt, reqno]{amsart}
\usepackage{amssymb}
\usepackage{eucal}
\usepackage{amsmath}
\usepackage{amscd}
\usepackage[dvips]{color}
\usepackage{multicol}
\usepackage[all]{xy}           
\usepackage{graphicx}
\usepackage{color}
\usepackage{colordvi}
\usepackage{xspace}
\usepackage{ulem}
\usepackage{cancel}
\usepackage{tikz}
\usepackage{longtable}
\usepackage{multirow}
\usepackage{ifpdf}
\ifpdf
 \usepackage[colorlinks,final,hyperindex]{hyperref}
\else
 \usepackage[colorlinks,final,hyperindex,hypertex]{hyperref}
\fi

\begin{document}
\newcommand {\emptycomment}[1]{} 

\newcommand{\tabincell}[2]{\begin{tabular}{@{}#1@{}}#2\end{tabular}}

\baselineskip=14pt
\newcommand{\nc}{\newcommand}
\newcommand{\delete}[1]{}
\nc{\mfootnote}[1]{\footnote{#1}} 
\nc{\todo}[1]{\tred{To do:} #1}

\delete{
\nc{\mlabel}[1]{\label{#1}}  
\nc{\mcite}[1]{\cite{#1}}  
\nc{\mref}[1]{\ref{#1}}  
\nc{\meqref}[1]{\eqref{#1}} 
\nc{\bibitem}[1]{\bibitem{#1}} 
}

\nc{\mlabel}[1]{\label{#1}  
{\hfill \hspace{1cm}{\bf{{\ }\hfill(#1)}}}}
\nc{\mcite}[1]{\cite{#1}{{\bf{{\ }(#1)}}}}  
\nc{\mref}[1]{\ref{#1}{{\bf{{\ }(#1)}}}}  
\nc{\meqref}[1]{\eqref{#1}{{\bf{{\ }(#1)}}}} 
\nc{\mbibitem}[1]{\bibitem[\bf #1]{#1}} 

\newtheorem{thm}{Theorem}[section]
\newtheorem{lem}[thm]{Lemma}
\newtheorem{cor}[thm]{Corollary}
\newtheorem{pro}[thm]{Proposition}
\newtheorem{conj}[thm]{Conjecture}
\theoremstyle{definition}
\newtheorem{defi}[thm]{Definition}
\newtheorem{ex}[thm]{Example}
\newtheorem{rmk}[thm]{Remark}
\newtheorem{pdef}[thm]{Proposition-Definition}
\newtheorem{condition}[thm]{Condition}

\newcommand{\C}{\mathbb C}
\newcommand{\Q}{\mathbb{Q}}
\newcommand{\R}{\mathbb{R}}
\newcommand{\M}{\mathbb{M}}
\newcommand{\N}{\mathbb{N}}
\newcommand{\Z}{\mathbb{Z}}



\title[Anti-dendriform algebras]
{Anti-dendriform algebras, new splitting of operations and Novikov
type algebras}

\author{Dongfang Gao}
\address{Chern Institute of Mathematics \& LPMC, Nankai University, Tianjin 300071, China}
\email{gaodfw@mail.ustc.edu.cn}

\author{Guilai Liu}
\address{Chern Institute of Mathematics \& LPMC, Nankai University, Tianjin 300071, China}
\email{1120190007@mail.nankai.edu.cn}

\author{Chengming Bai}
\address{Chern Institute of Mathematics \& LPMC, Nankai University, Tianjin 300071, China }
\email{baicm@nankai.edu.cn}


\begin{abstract}
We introduce the notion of anti-dendriform algebras as a new
approach of splitting the associativity. They are characterized as
the algebras with two operations whose sum is associative and the
negative left and right multiplication operators compose the
bimodules of the sum associative algebras, justifying the notion
due to the comparison with the corresponding characterization of
dendriform algebras. The notions of anti-$\mathcal O$-operators
and anti-Rota-Baxter operators on associative algebras are
introduced to interpret anti-dendriform algebras. In particular,
there are compatible anti-dendriform algebra structures on
associative algebras with nondegenerate commutative Connes
cocycles. There is an important observation that there are
correspondences between certain subclasses of dendriform and
anti-dendriform algebras in terms of $q$-algebras. As a direct
consequence, we give the notion of Novikov-type dendriform
algebras as an analogue of Novikov algebras for dendriform
algebras, whose relationship with Novikov algebras is consistent
with the one between dendriform and pre-Lie algebras. Finally we
extend to provide a general framework of introducing the notions
of analogues of anti-dendriform algebras, which interprets a new
splitting of operations.

\end{abstract}


\subjclass[2010]{
17A36,  
17A40,  
17B10, 
17B40, 
17B60, 
17B63,  
17D25.  
}

\keywords{associative algebra; dendriform algebra; anti-dendriform
algebra; commutative Connes cocycle; Novikov algebra}

\maketitle


\tableofcontents

\allowdisplaybreaks

\section{Introduction}

The aim of this paper is to introduce the notion of
anti-dendriform algebras illustrating a new splitting of
operations, and study the relationships between them and the
related structures such as anti-$\mathcal O$-operators,
commutative Connes cocycles on associative algebras, dendriform
algebras and Novikov algebras.

Recall that a dendriform algebra is a vector space $A$ with two
bilinear operations $\succ,\prec$ satisfying
\begin{align}\label{dendriform algebra}
x\succ(y\succ z)=(x\cdot y)\succ z,\ \ (x\prec y)\prec z=x\prec(y\cdot z),\ \ (x\succ y)\prec z=x\succ (y\prec z),
\end{align}
where
\begin{align}\label{dendriform algebra'}
x\cdot y=x\succ y+x\prec y,
\end{align}
 for all $x,y,z\in A$. The
notion of dendriform algebras was introduced by Loday in the study
of algebraic K-theory (\cite{Lo}). They appear in a lot of fields
in mathematics and physics, such as arithmetic (\cite {Lo1}),
combinatorics (\cite{LR1}), Hopf algebras (\cite{Cha, Ho, Ho1,
LR2,Ro}), homology (\cite{Fra, Fra1}), operads (\cite{Lo2}), Lie
and Leibniz algebras (\cite{Fra1}) and quantum field theory
(\cite{Fo}). The fact that the sum of the two operations in a
dendriform algebra $(A,\succ,\prec)$ gives an associative algebra
$(A,\cdot)$ expresses a kind of ``splitting the associativity".
 Moreover, dendriform algebras are closely related to
pre-Lie algebras which are a class of Lie-admissible algebras
whose commutators are Lie algebras, also appearing in many fields in
mathematics and physics (\cite{Bai2,Bu} and the references
therein), in the sense that for a dendriform algebra
$(A,\succ,\prec)$, the bilinear operation
\begin{equation}\label{eq:pre}
x\ast  y=x\succ y-y\prec x,\;\;\forall x,y\in A,
\end{equation}
defines a pre-Lie algebra $(A,\ast)$, which is called the
associated pre-Lie algebra of $(A,\succ,\prec)$. Therefore there
is the following relationship among Lie algebras,
associative algebras, pre-Lie algebras and dendriform algebras in
the sense of commutative diagram of categories (\cite{Cha1}):
\begin{equation}\label{eq:dendri} \begin{matrix} {\rm dendriform\quad algebras} & \longrightarrow & \mbox{pre-Lie algebras} \cr \downarrow & &\downarrow\cr {\rm
associative\quad algebras} & \longrightarrow & {\rm Lie\quad
algebras.} \cr\end{matrix}\end{equation}

On the other hand, there is an ``anti-structure" for pre-Lie
algebras, namely anti-pre-Lie algebras, introduced in \cite{LB},
which are characterized as the Lie-admissible algebras whose
negative left multiplication operators give representations of the
commutator Lie algebras, justifying the notion since pre-Lie
algebras are the Lie-admissible algebras whose left multiplication
operators give representations of the commutator Lie algebras.

There is a new approach of splitting operations, motivated by the
study of anti-pre-Lie algebras. We introduce the notion of
anti-dendriform algebras, still keeping the property of splitting
the associativity, but it is the negative left and right
multiplication operators that compose the bimodules of the sum
associative algebras, instead of the left and right multiplication
operators doing so for dendriform algebras. Such a
characterization justifies the notion, and moreover, the following
commutative diagram holds, which is the diagram~(\ref{eq:dendri})
with replacing dendriform and pre-Lie algebras by anti-dendriform
and anti-pre-Lie algebras respectively.

\begin{equation}\begin{matrix} \mbox{ anti-dendriform algebras} & \longrightarrow & \mbox{anti-pre-Lie algebras} \cr \downarrow & &\downarrow\cr {\rm
associative\quad algebras} & \longrightarrow & {\rm Lie\quad
algebras} \cr\end{matrix}\label{diag:comm}\end{equation}

\delete{ In this paper, we study the commutative Connes cocycle on
an associative algebra $A,$ that is, a symmetric bilinear form
$\mathcal{B}(\ , \ )$ satsifying Eq.~\eqref{circulation}.
We want to investigate that ``what algebra structure can be
obtained from a non-degenerate commutative Connes cocycle on an
associative algebra?'' and ``what are the difference between them
and dendriform algebras?'' Here the algebra structure is called
the anti-dendriform algebra. Recall that for the dendriform
algebra $(A,\succ,\prec),$ we know that the bilinear operation
\begin{align}\label{intro-2}
x\ast y=x\succ y-y\prec x, \quad \forall x,y\in A,
\end{align}
defines a pre-Lie algebra structure on $A,$ the bilinear operation
\begin{align}\label{intro-1}
x\cdot y=x\succ y+x\prec y, \quad \forall x,y\in A,
\end{align}
defines an associative algebra structure on $A$ and the triple $(A, L_\succ, R_\prec)$ is a bimodule of the associative algebra $(A,\cdot).$
For the anti-dendriform algebras, what algebra structures do the
bilinear operations \eqref{intro-2} and \eqref{intro-1} define on
them respectively? Furthermore, there is a compatible dendriform
algebra structure on an associative algebra $A$ if and only if
there exists an invertible $\mathcal{O}$-operator of $A$
associated to certain suitable bimodule of $A$ (see \cite{BGN1}).
For the anti-dendriform algebra, whether it has the similar
properties. }

As $\mathcal O$-operators and Rota-Baxter operators on associative
algebras interpreting dendriform algebras (\cite{BGN2}), we
introduce the notions of anti-$\mathcal O$-operators and
anti-Rota-Baxter operators on associative algebras to interpret
anti-dendriform algebras. In particular,  there are compatible
anti-dendriform algebra structures on associative algebras with
nondegenerate commutative Connes cocycles.

In \cite{LB}, there is an important observation that there is a
correspondence between Novikov algebras as a subclass of pre-Lie
algebras and admissible Novikov algebras as a subclass of
anti-pre-Lie algebras in terms of $q$-algebras. That is, the
$2$-algebra of a Novikov algebra is an admissible Novikov algebra,
whereas the $-2$-algebra of an admissible Novikov algebra is a
Novikov algebra. We also find there is a similar correspondence
between some subclasses of dendriform algebras and anti-dendriform
algebras in terms of $q$-algebras. Note that such a correspondence
is available for any $q\ne 0, \pm 1$, not for only a special value of $q$ in \cite{LB}, which in fact corresponds
to $q=-2$ in this paper.  We also extend the correspondence between the
subclasses of pre-Lie algebras and anti-pre-Lie algebras for these
$q$s and in particular, for a fixed $q\ne 0,\pm 1$,  the relationship between the
corresponding subclasses of dendriform algebras and pre-Lie
algebras as well as anti-dendriform algebras and anti-pre-Lie
algebras is still kept as the one given by Eq.~(\ref{eq:pre}).

Moreover, there is an interesting byproduct. As a subclass of
pre-Lie algebras, Novikov algebras were introduced in connection
with Hamiltonian operators in the formal variational calculus
(\cite{Gel}) and Poisson brackets of hydrodynamic type
(\cite{Bal}). On the other hand, both pre-Lie algebras and
dendriform algebras are examples of splitting operations and their
operads are the successors of operads of Lie and associative
algebras respectively (\cite{BBGN}). So it is natural to ask
whether and how one can give a reasonable notion of analogues
of Novikov algebras for the successors' algebras, in particular,
for dendriform algebras? In fact, the above approach answers this
problem. Due to the introduction of the notion of anti-dendriform
algebras and the above correspondence, one might introduce the
notion of Novikov-type dendriform algebras as the aforementioned
subclass of dendriform algebras for $q=-2$. The speciality of
$q=-2$ also can be seen from the identity involving $q$
(Proposition~\ref{dendriform-anti-dendriform}). Moreover, it is
consistent with the diagram~(\ref{eq:dendri}) in the following
sense:
\begin{equation}\label{eq:N-dendri} \begin{matrix} {\rm Novikov\text{-}type\; dendriform\; algebras} & \longrightarrow & \mbox{Novikov algebras} \cr \downarrow & &\downarrow\cr {\rm
associative\quad algebras} & \longrightarrow & {\rm Lie\quad
algebras.} \cr\end{matrix}\end{equation} We would like to point out
that this ``rule" of constructing analogues of Novikov algebras
for dendriform algebras is due to the introduction of the notion
of anti-dendriform algebras and hence it is regarded as an
application of the latter.

The paper is organized as follows. In Section 2, we introduce the
notion of anti-dendriform algebras as a new approach of splitting
the associativity. The notions of anti-$\mathcal O$-operators and
anti-Rota-Baxter operators on associative algebras are introduced
to interpret anti-dendriform algebras. The relationships between
anti-dendriform algebras and commutative Connes cocycles on
associative
 algebras are given. In Section 3, we investigate the correspondences of some subclasses
of dendriform algebras and anti-dendriform algebras as well as
pre-Lie algebras and anti-pre-Lie algebras in terms of
$q$-algebras. The relationships among these subclasses are given.
In particular, in the case that
$q=-2$, we introduce the notions of Novikov-type dendriform algebras and
admissible Novikov-type dendriform algebras with their correspondences.
 In Section 4,  we provide a general framework of introducing the
notions of analogues of anti-dendriform algebras to interpret a
new splitting of operations. They are
characterized in terms of double spaces.


Throughout this paper, all vector spaces are assumed to be
finite-dimensional over a field $\mathbb{F}$ of characteristic 0,
although many results are still available in the
infinite-dimensional case.

\section{Anti-dendriform algebras}
We introduce the notion of anti-dendriform algebras as a new
approach of splitting the associativity, characterized as the
associative admissible algebras whose negative left and right
multiplication operators compose the bimodules of the associated
associative algebras. We introduce the notions of anti-$\mathcal
O$-operators and anti-Rota-Baxter operators on associative
algebras to interpret anti-dendriform algebras. There is a
compatible anti-dendriform algebra structure on an associative
algebra if and only if there exists an invertible
anti-$\mathcal{O}$-operator of the associative algebra. In
particular, there are compatible anti-dendriform algebra
structures on associative algebras with nondegenerate commutative
Connes cocycles.


\subsection{Anti-dendriform algebras}

\begin{defi}\label{defi:asso admissible algebras}
Let $A$ be a vector space with two bilinear operations
$$\vartriangleright: A\otimes A\rightarrow A,\quad\vartriangleleft: A\otimes A\rightarrow A.$$
Define a bilinear operation $\cdot$ as
\begin{equation}\label{eq:asso}
x\cdot y=x\vartriangleright y+x\vartriangleleft y,\;\;\forall x,y\in A.
\end{equation}
The triple $(A,\vartriangleright,\vartriangleleft)$ is called an
{\bf associative admissible algebra} if $(A, \cdot)$ is an
associative algebra. In this case, $(A,\cdot)$ is called the {\bf associated
associative algebra} of $(A,\vartriangleright,\vartriangleleft)$.
\end{defi}

\begin{rmk}\label{asso}
The triple $(A,\vartriangleright,\vartriangleleft)$ is an associative admissible algebra if and only if the following equation holds:
\begin{align}
&(x\vartriangleright y)\vartriangleright z+(x\vartriangleleft y)\vartriangleright z+(x\vartriangleright y)\vartriangleleft z+(x\vartriangleleft y)\vartriangleleft z\notag\\
&=x\vartriangleright (y\vartriangleright z)+x\vartriangleright(y\vartriangleleft z)+x\vartriangleleft (y\vartriangleright z)+x\vartriangleleft (y\vartriangleleft z),\;\;\forall x,y,z\in A.\label{anti-admissible algebra}
\end{align}
\end{rmk}

It is known (\cite{Lo}) that dendriform algebras are associative
admissible algebras.

\begin{defi}\label{defi:anti-dendriform algebras}
Let $A$ be a vector space with two bilinear operations $\vartriangleright$ and $\vartriangleleft$.
The triple $(A,\vartriangleright,\vartriangleleft)$ is called an \textbf{anti-dendriform algebra} if the following equations hold:
\begin{equation}\label{eq:defi:anti-dendriform algebras1}
x\vartriangleright(y\vartriangleright z)=-(x\cdot y)\vartriangleright z=-x\vartriangleleft(y\cdot z)=(x\vartriangleleft y)\vartriangleleft z,
\end{equation}
\begin{equation}\label{eq:defi:anti-dendriform algebras2}
(x\vartriangleright y)\vartriangleleft z=x\vartriangleright (y\vartriangleleft z),\;\;\forall x,y,z
\in A,
\end{equation}
where the bilinear operation $\cdot$ is defined by Eq.~(\ref{eq:asso}).
\end{defi}

\begin{ex}\label{one-dim} Let $(A
,\vartriangleright,\vartriangleleft)$ be an 1-dimensional
anti-dendriform algebra with a basis $\{e\}$. Assume that
$$e\vartriangleright e=\alpha e,\quad e\vartriangleleft
e=\beta e,$$ where $\alpha,\beta\in\mathbb{F}$. Then by
Eq.~(\ref{eq:defi:anti-dendriform algebras1}), we have
$$\alpha^2e=(-\alpha^2-\alpha\beta) e=(-\beta^2-\alpha\beta)
e=\beta^2e.$$ Hence $\alpha=\beta=0$, that is, any 1-dimensional
anti-dendriform algebra is trivial.
\end{ex}

Recall that $(A,\circ)$ is called a \textbf{Lie-admissible
algebra}, where $A$ is a vector space with a bilinear operation
$\circ: A\otimes A\rightarrow A$, if the bilinear operation
$[\,,\,]:A\otimes A\rightarrow A$ defined by
\begin{equation} [x,y]=x\circ y-y\circ x,\;\; \forall x,y\in A,\label{eq:Lie}
\end{equation}
makes $(A,[\,,\,])$ a Lie algebra, which is called the \textbf{sub-adjacent Lie algebra} of
$(A,\circ)$ and denoted by $(\frak g(A),[\,,\,])$.
Obviously, an associative algebra is a Lie-admissible algebra.

An {\bf anti-pre-Lie algebra} (\cite{LB}) is a vector space $A$ with a bilinear operation $\circ$ satisfying
\begin{align}
  &x\circ(y\circ z)-y\circ(x\circ z)=[y,x]\circ z, \label{eq:21}\\
  &[x,y]\circ z+[y,z]\circ x+[z,x]\circ y=0, \;\;\forall x,y,z\in
  A,\label{eq:22}
\end{align}
where the bilinear operation $[\,,\,]$ is defined by
Eq.~(\ref{eq:Lie}). Equivalently, an anti-pre-Lie algebra
$(A,\circ)$ is a Lie-admissible algebra satisfying Eq.~(\ref{eq:21}).


\begin{pro}\label{property-1}
Let $(A, \vartriangleright, \vartriangleleft)$ be an anti-dendriform algebra. 
\begin{enumerate}
 \item\label{it:1} Define a bilinear operation $\cdot$ by Eq.~\eqref{eq:asso}. Then $(A,\cdot)$ is an associative algebra,
 called the {\bf associated associative algebra} of $(A, \vartriangleright, \vartriangleleft).$
 Furthermore, $(A, \vartriangleright, \vartriangleleft)$ is called a {\bf compatible anti-dendriform algebra structure} on $(A,\cdot).$
   \item\label{it:2} The bilinear operation $\circ:A\otimes A\rightarrow A$ given by \begin{equation}\label{eq:antipl1}x\circ y=x\vartriangleright y-y\vartriangleleft x,\ \ \forall x,y\in
   A,\end{equation}
  defines an anti-pre-Lie algebra, called the {\bf associated anti-pre-Lie algebra} of $(A, \vartriangleright, \vartriangleleft).$
  \item\label{it:3} Both $(A, \cdot)$ and $(A, \circ)$ have the same sub-adjacent Lie algebra $(\frak g(A),[\;,\,])$ defined by
\begin{equation}[x, y]=x\vartriangleright y+x\vartriangleleft
y-y\vartriangleright x-y\vartriangleleft x,\quad \forall x,y\in
A.\end{equation}
\end{enumerate}
\end{pro}

\begin{proof}

(\ref{it:11}). Obviously  Eq.~\eqref{anti-admissible algebra} follows from Eqs.~\eqref{eq:defi:anti-dendriform algebras1} and ~\eqref{eq:defi:anti-dendriform algebras2}.
Hence $(A,\cdot)$ is an associative algebra by Remark \ref{asso}.

(\ref{it:2}). Let $x,y,z\in A$. Then we have
 \begin{eqnarray*}
      x\circ (y\circ z)
                     &=&x\vartriangleright(y\vartriangleright z-z\vartriangleleft y)-(y\vartriangleright z-z\vartriangleleft y)\vartriangleleft x  \\
                     &=&x\vartriangleright(y\vartriangleright z)-x\vartriangleright(z\vartriangleleft y)-(y\vartriangleright z)\vartriangleleft x+(z\vartriangleleft y)\vartriangleleft x,\\
      (y\circ x)\circ z
                     &=&(y\vartriangleright x-x\vartriangleleft y)\vartriangleright z-z\vartriangleleft (y\vartriangleright x-x\vartriangleleft y) \\
                     &=&(y\vartriangleright x)\vartriangleright z-(x\vartriangleleft y)\vartriangleright z-z\vartriangleleft(y\vartriangleright x)+z\vartriangleleft(x\vartriangleleft y).
    \end{eqnarray*}
By swapping $x$ and $y,$ we have
 \begin{align*}
 y\circ (x\circ z)&=y\vartriangleright(x\vartriangleright z)-y\vartriangleright(z\vartriangleleft x)-(x\vartriangleright z)\vartriangleleft y+(z\vartriangleleft x)\vartriangleleft y,\\
 (x\circ y)\circ z&=(x\vartriangleright y)\vartriangleright z-(y\vartriangleleft x)\vartriangleright z-z\vartriangleleft(x\vartriangleright y)+z\vartriangleleft(y\vartriangleleft x).
 \end{align*}
Using Eqs.~\eqref{eq:defi:anti-dendriform algebras1} and~\eqref{eq:defi:anti-dendriform algebras2}, we obtain
\begin{align}
 x\circ (y\circ z)-y\circ (x\circ z)=&x\vartriangleright(y\vartriangleright z)+(z\vartriangleleft y)\vartriangleleft x-y\vartriangleright(x\vartriangleright z)-(z\vartriangleleft x)\vartriangleleft y\notag \\
                                =&(y\vartriangleright x+y\vartriangleleft x)\vartriangleright z-(x\vartriangleright y+x\vartriangleleft y)\vartriangleright z\notag \\
                                &-z\vartriangleleft(y\vartriangleright x+y\vartriangleleft x)+z\vartriangleleft (x\vartriangleright y+x\vartriangleleft y) \notag\label{anti-pre Lie 1}\\
                                =&(y\circ x)\circ z-(x\circ y)\circ z\notag=[y, x]\circ z\notag.
 \end{align}
Moreover, we have
$$x\circ y-y\circ x=x\vartriangleright y-y\vartriangleleft x-y\vartriangleright x+x\vartriangleleft y=x\cdot y-y\cdot x,\;\;\forall x,y\in A.$$
Thus $(A,\circ)$ is a Lie-admissible algebra and hence an anti-pre-Lie algebra.


\delete{From the above computations, we have $$[y, x]\circ
z=x\vartriangleright(y\vartriangleright z)+(z\vartriangleleft
y)\vartriangleleft x-y\vartriangleright(x\vartriangleright
z)-(z\vartriangleleft x)\vartriangleleft y.$$ By permuting
$y,x,z$,  we have
\begin{align*}
[x, z]\circ y=&z\vartriangleright(x\vartriangleright y)+(y\vartriangleleft x)\vartriangleleft z-x\vartriangleright(z\vartriangleright y)-(y\vartriangleleft z)\vartriangleleft x,\\
[z, y]\circ x=&y\vartriangleright(z\vartriangleright x)+(x\vartriangleleft z)\vartriangleleft y-z\vartriangleright(y\vartriangleright x)-(x\vartriangleleft y)\vartriangleleft z.
\end{align*}
So by Eq.~\eqref{eq:defi:anti-dendriform algebras1}, we have
\begin{equation*}\label{anti-pre Lie 2}
[y, x]\circ z+[x, z]\circ y+[z, y]\circ x=0.
\end{equation*}}

(\ref{it:3}). It is straightforward.
Note that it also appears in the proof of
Item~(\ref{it:2}).
\end{proof}

As a direct consequence, we have the following conclusion.

\begin{cor}
The commutative diagram~ {\rm (\ref{diag:comm})} holds.
\end{cor}

Recall that an associative algebra $(A,\cdot)$ is 2-nilpotent if $(x\cdot y)\cdot z=x\cdot (y\cdot z)=0$ for all
$x,y,z\in A$. 

\begin{pro}
Let $(A,\cdot)$ be a 2-nilpotent associative algebra. Then
$(A, \vartriangleright, \vartriangleleft)$ is a compatible anti-dendriform algebra if $\vartriangleright=\cdot, \vartriangleleft=0$ or $\vartriangleright=0, \vartriangleleft=\cdot$. Conversely, let $(A, \vartriangleright, \vartriangleleft)$ be an anti-dendriform algebra. If  $\vartriangleleft=0$ or $\vartriangleright=0$,
 then the associated associative algebra is 2-nilpotent.
\end{pro}

\begin{proof}
It is straightforward.
\end{proof}

\begin{pro}\label{idempotent}
Let $(A,\cdot)$ be an associative algebra with a non-zero idempotent $e$, that is, $e\cdot e=e$. Then there does not exist a compatible anti-dendriform algebra structure on $(A,\cdot)$.
\end{pro}

\begin{proof}
Assume that $(A, \vartriangleright, \vartriangleleft)$ is a compatible anti-dendriform algebra structure on $(A,\cdot)$. Then by Eq.~\eqref{eq:defi:anti-dendriform algebras1} for $e,e,e$, we have
 $$e\vartriangleright e=(e\cdot e)\vartriangleright e=e\vartriangleleft (e\cdot e)=e\vartriangleleft e.$$
On the other hand, note that $e=e\cdot e= e\vartriangleright
e+e\vartriangleleft e$. Hence $e\vartriangleright
e=e\vartriangleleft e=\frac{1}{2}e$. Then by
Eq.~\eqref{eq:defi:anti-dendriform algebras1} for $e,e,e$ again,
we have
\begin{equation}\label{idempotent-1}
\frac{1}{4}e=e\vartriangleright (e\vartriangleright e)=-(e\cdot e)\vartriangleright e=-\frac{1}{2}e,
\end{equation}
which is  a contradiction.
Hence the conclusion holds.
\end{proof}

It is known that any finite-dimensional associative algebra  without a non-zero idempotent element  is nilpotent. Therefore we have the following conclusion.

\begin{cor}
The associated associative algebra of any anti-dendriform algebra is nilpotent.
\end{cor}

\begin{ex}\label{ex:2-dim}
Let $(A,\cdot)$ be 2-dimensional nilpotent associative algebra over the complex field $\mathbb C$ with a basis $\{e_1,e_2\}$. Then it is known (for example see \cite{BM} or \cite{Bu1}) that $(A,\cdot)$ is isomorphic to
one of the following cases (only non-zero products are given):
\begin{enumerate}
\item [$(A1)_{\mbox{\ }}$] $(A,\cdot)$ is trivial, that is, all products are zero;
\item [$(A2)_{\mbox{\ }}$]$e_1\cdot e_1=e_2$.
\end{enumerate}
Obviously, both them are 2-nilpotent associative algebras.
Assume that $(A, \vartriangleright, \vartriangleleft)$ is a compatible anti-dendriform algebra structure on $(A,\cdot)$. Set
$$e_i \vartriangleright e_j=\alpha_{ij}e_1+\beta_{ij}e_2,\;\;\alpha_{ij},\beta_{ij}\in\mathbb{C},\;\; 1\leq i,j\leq 2.$$

\begin{enumerate}
\item[(I)] \label{it:1}$(A,\cdot)$ is $(A1)$. Then we have
$$e_i \vartriangleleft e_j=-\alpha_{ij}e_1-\beta_{ij}e_2,\;\; 1\leq i,j\leq 2.$$

\begin{enumerate}
\item[Case (1)] $\alpha_{22}=0$. Then $\mathbb{C}e_2$ is an 1-dimensional subalgebra of $(A, \vartriangleright, \vartriangleleft)$.
By Example \ref{one-dim}, we have $\beta_{22}=0$.
By Eq.~\eqref{eq:defi:anti-dendriform algebras1} for $e_1,e_2,e_2$ and $e_2,e_2,e_1$ respectively, we have
$$\alpha_{12}^2 e_1+\alpha_{12}\beta_{12}e_2=\alpha_{21}^2 e_1+\alpha_{21}\beta_{21}e_2=0.$$
Thus $\alpha_{12}=\alpha_{21}=0$. By
Eq.~(\ref{eq:defi:anti-dendriform algebras1}) for $e_1,e_1,e_2$ and
$e_2,e_1,e_1$ respectively, we have $\beta_{12}=\beta_{21}=0$.
By Eq.~(\ref{eq:defi:anti-dendriform algebras1}) for $e_1,e_1,e_1$, we have $\alpha_{11}=0$.

\item[Case (2)] $\beta_{11}=0$. Then by the linear transformation $e_1\rightarrow e_2, e_2\rightarrow e_1$, we get Case (1).

\item[Case (3)] $\beta_{11}\ne 0, \alpha_{22}\ne 0$.
By Eq.~(\ref{eq:defi:anti-dendriform algebras2}) for $e_1,e_1,e_2$,
we have
\begin{align*}
&\alpha_{11}\alpha_{12}+\beta_{11}\alpha_{22}=\alpha_{11}\alpha_{12}+\beta_{12}\alpha_{12},\ \ \alpha_{11}\beta_{12}+\beta_{11}\beta_{22}=\alpha_{12}\beta_{11}+\beta_{12}^2.
\end{align*}
Hence $\alpha_{12}\ne 0, \beta_{12}\ne 0$.
By Eq.~(\ref{eq:defi:anti-dendriform algebras1}) for $e_1,e_1,e_1$ and $e_2,e_2,e_2$ respectively,
we have
\begin{eqnarray*}
\alpha_{11}^2+\beta_{11}\alpha_{12}&=&\alpha_{11}^2+\beta_{11}\alpha_{21}=\beta_{11}(\alpha_{11}+\beta_{12})=\beta_{11}(\alpha_{11}+\beta_{21})=0,\label{e1e1e1}\\
(\alpha_{21}+\beta_{22})\alpha_{22}&=&(\alpha_{12}+\beta_{22})\alpha_{22}=\beta_{21}\alpha_{22}+\beta_{22}^2=\alpha_{22}\beta_{12}+\beta_{22}^2=0.\label{e1e1e2}
\end{eqnarray*}
Therefore  we have
$$\alpha_{12}=\alpha_{21}=-\beta_{22}=-\frac{\alpha_{11}^2}{\beta_{11}}\ne 0,\ \ \beta_{12}=\beta_{21}=-\alpha_{11},\ \ \alpha_{22}=\frac{\alpha_{11}^3}{\beta_{11}^2}.$$
Hence by a straightforward computation, we have
$$(\frac{\alpha_{11}}{\beta_{11}}e_1+e_2)\vartriangleright(\frac{\alpha_{11}}{\beta_{11}}e_1+e_2)=(\frac{\alpha_{11}}{\beta_{11}}e_1+e_2)\vartriangleleft(\frac{\alpha_{11}}{\beta_{11}}e_1+e_2)=0.$$
Thus by the linear transformation $e_1\rightarrow e_1, e_2\rightarrow \frac{\alpha_{11}}{\beta_{11}}e_1+e_2$, we get Case (1).
\end{enumerate}
Obviously, $(A, \vartriangleright, \vartriangleleft)$ with the non-zero products given by
$$e_1\vartriangleright e_1=\gamma e_2,\ \ e_1\vartriangleleft e_1=-\gamma e_2,\;\;\gamma\in \mathbb C,$$
is an anti-dendriform algebra, corresponding to the above Case (1) with $\beta_{11}=\gamma$. Moreover it is straightforward to show that
these anti-dendriform algebras are classified up to isomorphism into the following two cases (only non-zero operations are given):
\begin{enumerate}
    \item [$(A1)_1$] $(A ,\vartriangleright,\vartriangleleft)$ is trivial;
    \item [$(A1)_2$] $e_1\vartriangleright e_1=e_2,\ \ e_1\vartriangleleft e_1=-e_2$.
\end{enumerate}
\item[(II)] $(A,\cdot)$ is $(A2)$. Then we have
\begin{align*}
&e_1\vartriangleleft e_1=-\alpha_{11} e_1-(\beta_{11}-1) e_2,\ \  e_1\vartriangleleft e_2=-\alpha_{12} e_1-\beta_{12} e_2,\\
&e_2\vartriangleleft e_1=-\alpha_{21} e_1-\beta_{21} e_2,\ \ e_2\vartriangleleft e_2=-\alpha_{22} e_1-\beta_{22} e_2.
\end{align*}
\begin{enumerate}
\item[Case (1)] $\alpha_{22}=0$. Then by a similar discussion as for Case (1) of (I), we have
$$\alpha_{11}=\alpha_{12}=\alpha_{21}=\beta_{12}=\beta_{21}=\beta_{22}=0.$$
\item[Case (2)] $\alpha_{22}\ne 0$. By Eq.~(\ref{eq:defi:anti-dendriform algebras1}) for $e_2,e_2,e_2$,
we have
$$\beta_{22}^2+\alpha_{22}\beta_{21}=\beta_{22}^2+\beta_{12}\alpha_{22}=\alpha_{22}(\alpha_{21}+\beta_{22})=\alpha_{22}(\alpha_{12}+\beta_{22})=0.$$
Thus we have
\begin{equation*}\label{e2e2e2-2}
\beta_{22}^2+\alpha_{22}\beta_{21}=0, \ \ \alpha_{12}=\alpha_{21}=-\beta_{22},\ \  \beta_{12}=\beta_{21}.
\end{equation*}
By Eq.~(\ref{eq:defi:anti-dendriform algebras1}) for $e_1,e_1,e_1$, we have
\begin{equation*}\label{e1e1e1-2}
\alpha_{21}=-\alpha_{12},\;\;\beta_{21}=-\beta_{12}.
\end{equation*}
Therefore we have
\begin{align}\label{e1e1e1-21}
\alpha_{12}=\alpha_{21}=\beta_{12}=\beta_{21}=\beta_{22}=0.
\end{align}
Hence by Eq.~(\ref{eq:defi:anti-dendriform algebras1}) for $e_1,e_1,e_2$, we have
$$-e_1\vartriangleright (e_1\vartriangleright e_2)=(e_1\vartriangleright e_1+e_1\vartriangleleft e_1)\vartriangleright e_2=\alpha_{22}e_1+\beta_{22}e_2=0.$$
Thus $\alpha_{22}=0$, which is a contradiction.
\end{enumerate}
Obviously, $(A, \vartriangleright, \vartriangleleft)$ with the non-zero products given by
$$e_1\vartriangleright e_1=\gamma e_2,\ \ e_1\vartriangleleft e_1=(1-\gamma) e_2,\;\;\gamma\in \mathbb C,$$
is an anti-dendriform algebra, corresponding to the above Case (1) with $\beta_{11}=\gamma$. Moreover it is straightforward to show that
these anti-dendriform algebras are classified up to isomorphism into the following cases (only non-zero operations are given):
\begin{enumerate}
    \item [$(A2)_{1\;\;\ \mbox{}}$] $e_1\vartriangleleft e_1=e_2$;
    \item [$(A2)_{2,\lambda}$] $e_1\vartriangleright e_1=e_2,\ \ e_1\vartriangleleft e_1=\lambda e_2$, where $\lambda\in\mathbb{C}$ with $\lambda\ne -1$.
\end{enumerate}
\end{enumerate}
In a summary, any 2-dimensional complex anti-dendriform algebra $(A ,\vartriangleright,\vartriangleleft)$ is isomorphic to one of the following mutually
non-isomorphic cases (only non-zero products are given):
\begin{enumerate}
\item [$(B1)_{\mbox{\ }}$] $(A ,\vartriangleright,\vartriangleleft)$ is trivial;
\item [$(B2)_{\mbox{\ }}$] $e_1\vartriangleleft e_1=e_2$;
\item [$(B3)_{\lambda}$] $e_1\vartriangleright e_1=e_2,\ \ e_1\vartriangleleft e_1=\lambda e_2$, where $\lambda\in\mathbb{C}$.
\end{enumerate}
Obviously, these anti-dendriform algebras are ``2-nilpotent" in the sense that all products involving three elements
such as $(x\vartriangleright y)\vartriangleright z$ and $x\vartriangleleft( y\vartriangleleft z)$
 are zero.
\end{ex}


\begin{ex}\label{3-dimension}
Let $(A,\cdot)$ be a 3-dimensional associative algebra with a basis $\{e_1,e_2,e_3\}$ whose nonzero products are given by
$$e_1\cdot e_1=e_2,\quad e_1\cdot e_2=e_2\cdot e_1=e_3.$$
By a straightforward computation,  $(A ,\vartriangleright,\vartriangleleft)$ is a compatible anti-dendriform algebra structure on $(A,\cdot)$ with the following non-zero products:
\begin{equation*}
e_1\vartriangleright e_1=\frac{1}{2}e_2+\gamma e_3,\;\; e_1\vartriangleleft e_1=\frac{1}{2}e_2-\gamma e_3,\;\;e_1\vartriangleright e_2=e_2\vartriangleleft e_1=2e_3,\; \; e_2\vartriangleright e_1=e_1\vartriangleleft e_2=-e_3.
\end{equation*}
 for any $\gamma\in\mathbb{F}$. Note that $(A ,\vartriangleright,\vartriangleleft)$ is not ``2-nilpotent" since $(e_1\vartriangleright e_1)\vartriangleleft e_1=e_3$.
\end{ex}

Let $(A,\cdot)$ be an associative algebra. Recall that a {\bf bimodule} of $(A,\cdot)$ is a triple $(V, l, r)$ consisting of a vector space $V$ and linear maps
$l,r: A\rightarrow \rm{End}_\mathbb{F}(V)$ such that
$$l(x\cdot y)v=l(x)(l(y)v), \ \ r(x\cdot y)v=r(y)(r(x)v), \ \ l(x)(r(y)v)=r(y)(l(x)v),\ \ \forall x,y\in A, v\in V.$$
In particular, $(A, L_\cdot, R_\cdot)$ is a bimodule of $(A,\cdot)$, where $L_\cdot,R_\cdot:A\rightarrow {\rm End}_\mathbb{F}(A)$ are two linear maps defined by
$L_\cdot(x)(y)=R_\cdot(y)(x)=x\cdot y$ for all $x,y\in A$ respectively.

Let $(A, \vartriangleright, \vartriangleleft)$ be an associative admissible algebra. Define two linear maps $L_\vartriangleright,R_\vartriangleleft:A\rightarrow {\rm End}_\mathbb{F}(A)$ respectively by
$$L_\vartriangleright(x)(y)=x\vartriangleright y,\quad R_\vartriangleleft(x)(y)=y\vartriangleleft x,\quad \forall x,y\in A.$$


\begin{pro}\label{anti-dendriform algebra equivalent}
Let $A$ be a vector space with two bilinear operations $\vartriangleright$ and $\vartriangleleft$. Define a bilinear operation $\cdot$ by Eq.~\eqref{eq:asso}.
Then the following conditions are equivalent.
\begin{enumerate}
  \item\label{it:11} $(A, \vartriangleright, \vartriangleleft)$ is an anti-dendriform algebra.
  \item\label{it:12} $(A, \vartriangleright, \vartriangleleft)$ is an associative admissible algebra,
  that is, $(A,\cdot)$ is an associative algebra, and for all $x, y, z\in A$, the following equations hold:
  \begin{equation}
 x\vartriangleright(y\vartriangleright z)=-(x\cdot y)\vartriangleright z,\ (x\vartriangleleft y)\vartriangleleft z=-x\vartriangleleft(y\cdot z),\
  (x\vartriangleright y)\vartriangleleft z=x\vartriangleright (y\vartriangleleft z).\label{anti-dendriform algebra equivalent-1}
  \end{equation}
  \item \label{it:13} $(A, \vartriangleright, \vartriangleleft)$ is an associative admissible algebra, that is, $(A,\cdot)$ is an associative algebra, and $(A, -L_\vartriangleright, -R_\vartriangleleft)$ is a bimodule of $(A,\cdot)$.
\end{enumerate}
\end{pro}

\begin{proof}
(\ref{it:11}) $\Longleftrightarrow$ (\ref{it:12}). It follows from Eqs.~\eqref{anti-admissible algebra},~\eqref{eq:defi:anti-dendriform algebras1} and ~\eqref{eq:defi:anti-dendriform algebras2}.

(\ref{it:12}) $\Longleftrightarrow$ (\ref{it:13}).
Let $x,y,z\in A$. Then we have
%
\begin{align*}
   (-L_\vartriangleright)(x\cdot y)(z)=-L_\vartriangleright(x)(-L_\vartriangleright(y)z)\ \ &\Longleftrightarrow \ \ x\vartriangleright (y\vartriangleright z)=-(x\cdot y)\vartriangleright z, \\
  (-R_\vartriangleleft)(x\cdot y)(z)=-R_\vartriangleleft(y)(-R_\vartriangleleft(x)z)\ \ &\Longleftrightarrow \ \ (z\vartriangleleft x)\vartriangleleft y=-z\vartriangleleft(x\cdot y), \\
  (-L_\vartriangleright)(x)(-R_\vartriangleleft(y)z)=(-R_\vartriangleleft(y))(-L_\vartriangleright(x)z) \ \ &\Longleftrightarrow \ \ x\vartriangleright(z\vartriangleleft y)=(x\vartriangleright z)\vartriangleleft y.
\end{align*}
Hence Item~(\ref{it:12}) holds if and only if Item~(\ref{it:13}) holds.
\end{proof}

\begin{rmk}
Recall (\cite{B}) that a dendriform algebra $(A,\succ,\prec)$ is
an associative admissible algebra such that $(A, L_\succ,
R_\prec)$ is a bimodule of the associated associative algebra
$(A,\cdot)$. Therefore the notion of anti-dendriform algebras is
justified due to the equivalent characterization (3) above.
\end{rmk}

Suppose that $(A,\cdot)$ is an associative algebra. Let $V$ be a vector space and $l,r: A\rightarrow {\rm End}_{\mathbb F}(V)$ be linear maps.
Then $(V, l, r)$ is a bimodule of $(A,\cdot)$ if and only if there is an associative algebra structure on the direct sum $A\oplus V$ of vector spaces  with the following bilinear operation, still denoted by
$\cdot$:
$$(x,u)\cdot(y,v)=(x\cdot y,l(x)v+r(y)u),\quad \forall x,y\in A,u,v\in V.$$
We denote this associative algebra by $A\ltimes_{l,r}V$.

\begin{cor}\label{cor:anti-d} Let $A$ be a vector space with two bilinear operations $\vartriangleright,\vartriangleleft: A\otimes A\rightarrow A$.
Then on the direct sum $\hat A:=A\oplus A$ of vector spaces,
the following bilinear operation
\begin{equation}\label{eq:asso-double}
(x,a)\cdot (y,b)=(x\vartriangleright y+x\vartriangleleft y,
-x\vartriangleright b-a\vartriangleleft y),\;\;\forall x,y,a,b\in
A,
\end{equation}
makes an associative algebra $(\hat A,\cdot)$ if and only if
$(A,\vartriangleright,\vartriangleleft)$ is an anti-dendriform algebra.
\end{cor}
\begin{proof}
It is clear that $(\hat A,\cdot)$ is an associative algebra if and only if $(A, \vartriangleright, \vartriangleleft)$ is an associative admissible algebra,
and $(A, -L_\vartriangleright, -R_\vartriangleleft)$ is a bimodule
of the associated associative algebra, which is equivalent to the
fact that $(A, \vartriangleright, \vartriangleleft)$ is an
anti-dendriform algebra by Proposition \ref{anti-dendriform
algebra equivalent}.
\end{proof}

\subsection{Anti-$\mathcal{O}$-operators and anti-Rota-Baxter operators}

\begin{defi}\label{anti-O}
Let $(A,\cdot)$ be an associative algebra and $(V, l, r)$ be a
bimodule. A linear map $T: V\rightarrow A$ is called an {\bf
anti-$\mathcal{O}$-operator} of $(A,\cdot)$ associated to $(V, l,
r)$ if the following equation holds: \begin{equation}T(u)\cdot
T(v)=-T\big(l(T(u))v+r(T(v))u\big), \quad \forall u,v\in
V.\end{equation} Furthermore, $T$ is called {\bf strong} if
\begin{equation}l(T(u)\cdot T(v))w=r(T(v)\cdot T(w))u, \quad \forall u, v,
w\in V.\end{equation} In particular, an
anti-$\mathcal{O}$-operator $T$ of $(A,\cdot)$ associated to the
bimodule $(A,L_\cdot,R_\cdot)$ is called an {\bf anti-Rota-Baxter
operator}, that is, $T:A\rightarrow A$ is a linear map satisfying
\begin{equation}\label{eq:ANTIRBO}T(x)\cdot T(y)=-T\big(T(x)\cdot y+x\cdot T(y)\big), \quad
\forall x,y\in A.\end{equation} An anti-Rota-Baxter operator $T$
is called {\bf strong} if $T$ satisfies \begin{equation}T(x)\cdot
T(y)\cdot z=x\cdot T(y)\cdot T(z), \quad \forall x, y, z\in
A.\end{equation} In these cases, we also call $(A,T)$ an {\bf
anti-Rota-Baxter algebra} and a {\bf strong anti-Rota-Baxter
algebra} respectively.
\end{defi}

\begin{rmk}
Let $(A,\cdot)$ be an associative algebra and $(V,l,r)$ be a
bimodule. Recall that a linear map $T:V\rightarrow A$ is called an
{\bf $\mathcal{O}$-operator} of $(A,\cdot)$ associated to the
bimodule $(V,l,r)$ if $T$ satisfies
$$T(u)\cdot T(v)=T\big(l(T(u))v+r(T(v))u\big), \quad \forall u,v\in V.$$
The notion of $\mathcal O$-operators was introduced in \cite{BGN}
(also appeared independently in \cite{U}) as a natural
generalization of Rota-Baxter operators, which correspond to the
solutions of associative Yang-Baxter equations in $(A,\cdot)$
under certain conditions. The notion of anti-$\mathcal
O$-operators is justified due to the comparison between them.
\end{rmk}


\begin{rmk}
Let $(A,\cdot)$ be an associative algebra and $(V, l, r)$ be a
bimodule. A linear map $D: A\rightarrow V$ is called an {\bf
anti-1-cocycle} of $(A,\cdot)$ associated to $(V, l, r)$ if the
following equation holds:
$$D(x\cdot y)=-\big(l(x)D(y)+r(y)D(x)\big),\quad \forall x,y\in A.$$
Obviously, an invertible linear map $T:V\rightarrow A$ is an
anti-$\mathcal O$-operator if and only if $T^{-1}$ is an
anti-1-cocycle.
\end{rmk}

\begin{pro}\label{anti}
Let $(A,\cdot)$ be an associative algebra and $(V,l,r)$ be a
bimodule.
 Suppose that $T: V\rightarrow A$ is an anti-$\mathcal{O}$-operator of $(A,\cdot)$ associated to $(V, l,
 r)$.
  Define two bilinear operations $\vartriangleright,\vartriangleleft$ on $V$  respectively as
  \begin{align}\label{anti-O1}
  u\vartriangleright v=-l(T(u))v,\quad u\vartriangleleft v=-r(T(v))u,\quad \forall u,v\in V.
  \end{align}
  Then the following conclusions hold.
  \begin{enumerate}
 \item\label{it:31} For all $u,v,w\in V$, the following equations hold:
  \begin{equation}\label{anti-O2}
  u\vartriangleright(v\vartriangleright w)=-(u\cdot v)\vartriangleright w,\ (u\vartriangleleft v)\vartriangleleft w=-u\vartriangleleft(v\cdot w), \ (u\vartriangleright v)\vartriangleleft w=u\vartriangleright (v\vartriangleleft
  w),
  \end{equation}
where $u \cdot v=u\vartriangleright v +u\vartriangleleft v$.

 \item\label{it:32} $(V, \vartriangleright, \vartriangleleft)$ is an
anti-dendriform algebra if and only if $T$ is strong. In this
case, $T$ is a homomorphism of associative algebras from the
associated associative algebra $(V,\cdot)$ to $(A,\cdot)$.
Furthermore, there is an induced anti-dendriform algebra structure
on $T(V)=\{T(u)~|~u\in V\}\subseteq A$ given by
\begin{align}\label{induce}
T(u)\vartriangleright T(v)=T(u\vartriangleright v),\quad
T(u)\vartriangleleft T(v)=T(u\vartriangleleft v),\quad \forall
u,v\in V,
\end{align}
and $T$ is a homomorphism of anti-dendriform algebras.
\end{enumerate}
\end{pro}

\begin{proof}
(\ref{it:31}). Let $u,v,w\in V$. Then we have
\begin{eqnarray*}
  -u\vartriangleleft(v\cdot w)&=&-u\vartriangleleft(v\vartriangleright w+v\vartriangleleft w)=-r\bigg(T\big(l(T(v))w\big)\bigg)u-r\bigg(T\big(r(T(w))v\big)\bigg)u \\
                   &=&
                   r\big(T(v)\cdot T(w)\big)u=r(T(w))(r(T(v))u)=(u\vartriangleleft v)\vartriangleleft w.
\end{eqnarray*}
Similarly, we have $$u\vartriangleright(v\vartriangleright
w)=-(u\cdot v)\vartriangleright w,\quad (u\vartriangleright
v)\vartriangleleft w=u\vartriangleright (v\vartriangleleft w).$$
Thus Eq.~\eqref{anti-O2} holds.

(\ref{it:32}). From Item~(\ref{it:31}) and
Definition~\ref{defi:anti-dendriform algebras}, $(V,
\vartriangleright, \vartriangleleft)$ is an anti-dendriform
algebra if and only if the following equation holds:
$$l\big(T(u)\cdot T(v)\big)w=u\vartriangleright(v\vartriangleright w)=(u\vartriangleleft v)\vartriangleleft w=r\big(T(v)\cdot T(w)\big)u,\quad \forall u,v,w\in V,
$$
that is, $T$ is a strong anti-$\mathcal{O}$-operator. The other
results follow immediately.
 \delete{In this case,
\begin{align*}
T(u\cdot v)&=T(u\vartriangleright v+u\vartriangleleft v)=T(u\vartriangleright v)+T(u\vartriangleleft v) \\
                    &=-T\big(l(T(u))v\big)-T\big(r(T(v))u\big)\\
                    &=T(u)T(v),\ \ \forall u,v\in V,
\end{align*}
that shows $T$ is a homomorphism of associative algebras from the associated associative algebra $(V,\cdot)$ to $A.$

Furthermore, for any $u,v,w\in V,$ we can get
\begin{align*}
  T(u)\vartriangleright (T(v)\vartriangleright T(w))&=T(u\vartriangleright (v\vartriangleright w))  \\
                            &=-T(u\vartriangleright v)\vartriangleright T(w)-T(u\vartriangleleft v)\vartriangleright T(w) \\
                            &=-(T(u)\vartriangleright T(v))\vartriangleright T(w)-(T(u)\vartriangleleft T(v))\vartriangleright T(w) \\
                            &=-(T(u)\vartriangleright T(v)+T(u)\vartriangleleft T(v))\vartriangleright T(w).
\end{align*}
Similarly, we can obtain
\begin{align*}
 (T(u)\vartriangleleft T(v))\vartriangleleft T(w)&=-T(u)\vartriangleleft (T(v)\vartriangleright T(w)+T(v)\vartriangleleft T(w)),\\
 (T(u)\vartriangleright T(v))\vartriangleleft T(w)&=T(u)\vartriangleright (T(v)\vartriangleleft T(w)),\\
(T(u)\vartriangleright T(v))\vartriangleright T(w)&=T(u)\vartriangleleft (T(v)\vartriangleleft T(w)),
\end{align*}
which imply that $T(V)$ is an anti-dendriform algebra. This
completes the proof.}
\end{proof}

\begin{cor}
Let $(A,\cdot)$ be an associative algebra and $P$ be a strong
anti-Rota-Baxter operator. Then the triple
$(A,\vartriangleright,\vartriangleleft)$ is an anti-dendriform
algebra, where \begin{equation}\label{eq:cons} x\vartriangleright
y=-P(x)\cdot y,\ \ x\vartriangleleft y=-x\cdot P(y), \ \ \forall
x,y\in A.\end{equation} Conversely, if $P:A\rightarrow A$ is a
linear transformation on an associative algebra $(A,\cdot)$ such
that Eq.~\eqref{eq:cons} defines an anti-dendriform algebra, then
$P$ satisfies
\begin{equation}P(x)P(y)z=-P\big(P(x)y+xP(y)\big)z=-xP\big(P(y)z+yP(z)\big)=xP(y)P(z),\
\ \forall x,y,z\in A.\end{equation}In particular, if $${\rm
Ann}^L_A(A)=\{x\in A~|~x\cdot y=0,\forall y\in A\}=0,\;\;{or}\;\;
{\rm Ann}^R_A(A)=\{x\in A~|~y\cdot x=0,\forall y\in A\}=0,$$ then
$P$ is a strong anti-Rota-Baxter operator.
\end{cor}

\begin{proof}
The first half part follows from Proposition~\ref{anti} by letting
$(V,l,r)=(A,L_{\cdot},R_{\cdot})$. The second half part follows from
Definition~\ref{defi:anti-dendriform algebras}.
\end{proof}

\begin{ex}\label{two-dim ass algebra}
Let $(A,\cdot)$ be a complex associative  algebra with a basis $\{e_1,e_2\}$
whose non-zero products are given by
$$e_1\cdot e_1=e_1,\ \ e_1\cdot e_2=e_2.$$
Suppose that $P:A\rightarrow A$ is a linear map whose corresponding matrix
is given by $\left( \begin{matrix}\alpha_{11} &\alpha_{12}\cr
\alpha_{21}&\alpha_{22}\cr\end{matrix}\right)$ under the basis $\{e_1,e_2\}$.
Then by Eq.~(\ref{eq:ANTIRBO}), $P$ is an anti-Rota-Baxter
operator on $(A,\cdot)$ if and only if
$$\alpha_{11}=\alpha_{12}=\alpha_{22}=0.$$
Therefore the set of all anti-Rota-Baxter operators on $(A,\cdot)$ is $\{ P=\left(
\begin{matrix}0 &0 \cr
\gamma&0\cr\end{matrix}\right)|\gamma\in \mathbb C\}$. Moreover, any
anti-Rota-Baxter operator on $(A,\cdot)$ is strong. Hence by
Eq.~(\ref{eq:cons}), we obtain the following anti-dendriform
algebras whose non-zero products are given by
$$e_1\vartriangleleft e_1=-\gamma e_2,\;\; \gamma\in
\mathbb C.$$
It is straightforward to show that if $\gamma=0$, then it is isomorphic to $(B1)$ and
if $\gamma\ne 0$, then it is isomorphic to $(B2)$, where the notations are given in Example~\ref{ex:2-dim}.

\end{ex}

\delete{
\begin{pro}
Let $A$ be an associative algebra and $(A,l,r)$ be an
$A$-bimodule. Define two bilinear operations:
$$x\vartriangleright y=-l(x)y,\quad x\vartriangleleft y=-r(y) x,\quad \forall x,y\in A.$$
If $Ann_A(A)=\{x\in A~|~l(x) y=0,\forall y\in A\}=0,$ then $(A,\vartriangleright,\vartriangleleft)$ is an anti-dendriform algebra if and only if
$$xy=x\vartriangleright y+x\vartriangleleft y,\quad x,y\in A.$$
\end{pro}
\begin{proof}
$(\Rightarrow)$ Suppose that $(A,\vartriangleright,\vartriangleleft)$ is an anti-dendriform algebra. Note that for any $x,y,z\in A,$
\begin{align*}
l(xy)z=l(x)l(y)z=l(x)(-y\vartriangleright z)=x\vartriangleright(y\vartriangleright z)=
-(x\vartriangleright y+x\vartriangleleft y)\vartriangleright z=l(x\vartriangleright y+x\vartriangleleft y)z.
\end{align*}
So $xy=x\vartriangleright y+x\vartriangleleft y$ since $Ann_A(A)=0.$

$(\Leftarrow)$ It is easy to see that identity map $Id:A\longrightarrow A$ is an anti-$\mathcal{O}$-operator associated to $A$-bimodule $(A,l,r).$
So,  $$x\vartriangleright(y\vartriangleright z)=-(x\cdot y)\vartriangleright z,\ \ (x\vartriangleleft y)\vartriangleleft z=-x\vartriangleleft(y\cdot z),\ \ (x\vartriangleright y)\vartriangleleft z=x\vartriangleright (y\vartriangleleft z)$$
by Proposition \ref{anti}.
Moreover, $(A,\vartriangleright,\vartriangleleft)$ is an associative admissible algebra since $A$ is an associative algebra and
$xy=x\vartriangleright y+x\vartriangleleft y.$ Therefore $(A,\vartriangleright,\vartriangleleft)$ is an anti-dendriform algebra by Proposition \ref{anti-dendriform algebra equivalent}.
\end{proof}
}

\begin{lem}\label{automatically}
An invertible anti-$\mathcal{O}$-operator of an associative
algebra is automatically strong.
\end{lem}
\begin{proof}
Let $T:V\rightarrow A$ be an invertible
anti-$\mathcal{O}$-operator of an associative algebra
$(A,\cdot_A)$ associated to a bimodule $(V, l, r)$. Define two
bilinear operations $\vartriangleright,\vartriangleleft$ on $V$
respectively by Eq.~\eqref{anti-O1}.
 Define a bilinear operation $\cdot_V$ on $V$ by
$$u\cdot_V v=u\vartriangleright v+u\vartriangleleft v,\;\;\forall
u,v\in V.$$ Let $u,v,w\in V$. Then we have
\begin{align*}
  \big(T(u)\cdot_A T(v)\big)\cdot_A T(w)&=-T\big(l(T(u))v+r(T(v))u\big)\cdot_A T(w) \\
                        &=T\bigg(l\big(T(l(T(u))v+r(T(v))u)\big)w+r(T(w))(l(T(u))v+r(T(v))u)\bigg) \\
                        &=-T\big((l(T(u))v+r(T(v))u)\vartriangleright w+(l(T(u))v+r(T(v))u)\vartriangleleft w\big)\\
                        &=T\bigg((u\cdot_V v)\vartriangleright w+(u\vartriangleright v)\vartriangleleft w+(u\vartriangleleft v)\vartriangleleft
                        w\bigg).
\end{align*}
Similarly, we have
$$T(u)\cdot_A \big(T(v)\cdot_A T(w)\big)=T\bigg(u\vartriangleright(v\vartriangleright w)+u\vartriangleright(v\vartriangleleft w)+u\vartriangleleft(v\cdot_V w)\bigg).$$
Since $(A,\cdot_A)$ is an associative algebra and $T$ is
invertible, we have
$$(u\cdot_V v)\vartriangleright w+(u\vartriangleright v)\vartriangleleft w+(u\vartriangleleft v)\vartriangleleft w=u\vartriangleright(v\vartriangleright w)+u\vartriangleright(v\vartriangleleft w)+u\vartriangleleft(v\cdot_V w).$$
By Proposition \ref{anti}, Eq.~\eqref{anti-O2} holds and hence
$u\vartriangleright(v\vartriangleright w)=(u\vartriangleleft
v)\vartriangleleft w$. Therefore $(V,
\vartriangleright,\vartriangleleft)$ is an anti-dendriform algebra
and thus by Proposition \ref{anti} again,  $T$ is strong.
\end{proof}

\begin{thm}\label{anti-dendriform-invertible anti-O}
Let $(A,\cdot)$ be an associative algebra. Then there is a
compatible anti-dendriform algebra structure on $(A,\cdot)$ if and
only if there exists an invertible anti-$\mathcal{O}$-operator of
$(A,\cdot)$.
\end{thm}

\begin{proof}
Suppose that $(A,\vartriangleright,\vartriangleleft)$ is a compatible
anti-dendriform algebra structure on $(A,\cdot)$. Then
$$x\cdot
y=x\vartriangleright y+x\vartriangleleft y=-(-L_\vartriangleright
(x) y-R_\vartriangleleft (y)x),\;\;\forall x,y\in A.$$ Hence the
identity map ${\rm Id}:A\rightarrow A$ is an invertible
anti-$\mathcal{O}$-operator of $(A,\cdot)$ associated to the
bimodule $(A,-L_\vartriangleright,-R_\vartriangleleft)$.

Conversely, suppose that $T:V\rightarrow A$ is an invertible
anti-$\mathcal{O}$-operator of $(A,\cdot)$ associated to a
bimodule $(V, l, r)$ of $(A,\cdot)$. Then by Lemma
\ref{automatically} and Proposition \ref{anti}, there exist
anti-dendriform algebra structures on $V$ and $T(V)=A$ defined by
Eqs.~\eqref{anti-O1} and ~\eqref{induce} respectively. Let $x,y\in
A$. Then there exist $u,v\in V$ such that $x=T(u),y=T(v)$. Hence
we have
\begin{align*}
x\cdot y&=T(u)\cdot T(v)=-T(l(T(u))v+r(T(v))u)=T(u\vartriangleright v+u\vartriangleleft v)\\
&=T(u)\vartriangleright T(v)+T(u)\vartriangleleft T(v)=x\vartriangleright y+x\vartriangleleft y.
\end{align*}
So $(A,\vartriangleright,\vartriangleleft)$ is a compatible
anti-dendriform algebra structure on $(A,\cdot)$.
\end{proof}

\begin{pro}\label{embedding}
Let $(A,\cdot)$ be an associative algebra and $(V, l, r)$ be a
bimodule. Suppose that $T:V\longrightarrow A$ is a linear map.
Then $T$ is an anti-$\mathcal{O}$-operator of $(A,\cdot)$
associated to $(V, l, r)$ if and only if the linear map
$$\hat T: A\ltimes_{l,r} V\longrightarrow A\ltimes_{l,r} V,\quad (x,u)\longmapsto
(T(u),0),$$ is an anti-Rota-Baxter operator on the associative
algebra $A\ltimes_{l,r}V$.
\end{pro}
\begin{proof}
Let  $x,y\in A,u,v\in V$. Then we have
\begin{eqnarray*}
\hat T((x,u))\cdot \hat T((y,v))&=&(T(u),0)\cdot(T(v),0)=(T(u)\cdot T(v),0),\\
\hat T((x,u))\cdot(y,v)&=&(T(u),0)\cdot(y,v)=(T(u)\cdot y, l(T(u))v),\\
(x,u)\cdot \hat T ((y,v))&=&(x,u)\cdot (T(v),0)=(x\cdot T(v),
r(T(v))u).
\end{eqnarray*}
Hence $\hat T$ is an anti-Rota-Baxter operator on the associative
algebra $A\ltimes_{l,r}V$ if and only if
$$(T(u)\cdot T(v),0)=-\bigg(T\big(l(T(u))v+r(T(v))u\big),0\bigg),$$
that is, $T$ is an anti-$\mathcal{O}$-operator of $(A,\cdot)$
associated to $(V, l, r)$.
\end{proof}


\begin{cor}
Let $(A,\vartriangleright,\vartriangleleft)$ be an
anti-dendriform algebra and $(A,\cdot)$ be the associated
associative algebra. Set $\hat A=A\oplus A$ as the direct sum of
vector spaces. Define a bilinear operation $\cdot$ on $\hat A$ by
Eq.~\eqref{eq:asso-double} and a linear map $\widehat {\rm
Id}:\hat A\rightarrow \hat A$ by 
\begin{equation*}
\widehat {\rm Id}((x,y))=(y,0),\;\;\forall x,y\in A.
\end{equation*}
Then $\widehat {\rm Id}$ is an anti-Rota-Baxter operator on the
associative algebra $(\hat A,\cdot)$, that is, $(\hat A, \widehat
{\rm Id})$ is an anti-Rota-Baxter algebra.
\end{cor}

\begin{proof}
By Corollary~\ref{cor:anti-d}, $(\hat A,\cdot)$ is an associative
algebra, which is exactly
$A\ltimes_{-L_\vartriangleright,-R_\vartriangleleft} A$. Since
${\rm Id}:A\rightarrow A$ is an anti-$\mathcal{O}$-operator of
$(A,\cdot)$ associated to the bimodule
$(A,-L_\vartriangleright,-R_\vartriangleleft)$, by
Proposition~\ref{embedding}, $\widehat {\rm Id}$ is an
anti-Rota-Baxter operator on the associative algebra $(\hat
A
,\cdot)$.
\end{proof}

\begin{rmk}
In general, $\widehat {\rm Id}$ is not a strong anti-Rota-Baxter
operator on the associative algebra $(\hat A,\cdot)$ and hence one
shows that there is not an anti-dendriform algebra structure on
$\hat A$ defined by Eq.~(\ref{eq:cons}). On the other hand, we
still define two bilinear operations
$\vartriangleright,\vartriangleleft$ on the vector subspace
$A'=\{(0,x)|x\in A\}\subset \hat A$ by
\begin{equation*}
(0,x)\vartriangleright (0,y)=-\widehat {\rm Id}((0,x))\cdot
(0,y)=(0, x\vartriangleright y),\;\;(0,x)\vartriangleleft
(0,y)=-(0,x)\cdot \widehat {\rm Id}((0,y))=(0, x\vartriangleleft y),
\end{equation*}
for all $x,y\in A$. Then it is straightforward to show that
$(A',\vartriangleright,\vartriangleleft)$ is an anti-dendriform
algebra. That is, there is still an anti-dendriform algebra
structure on the subspace $A'$ of $\hat A$ defined by the
anti-Rota-Baxter operator $\widehat {\rm Id}$ through
Eq.~(\ref{eq:cons}). Let $F:A\rightarrow A'$ be a linear map
defined by
\begin{equation*}
F(x)=(0,x),\;\;\forall x\in A.
\end{equation*}
Then $F$ is an isomorphism of anti-dendriform algebras from
$(A,\vartriangleright,\vartriangleleft)$ to
$(A',\vartriangleright,\vartriangleleft)$. Hence in the sense
above, the anti-dendriform algebra
$(A,\vartriangleright,\vartriangleleft)$ is ``embedded" into the
anti-Rota-Baxter algebra $(\hat A, \widehat {\rm Id})$. Note that
it is a little different from the case of dendriform algebras and
Rota-Baxter algebras given in \cite{GP}, where there is a
dendriform algebra structure on the whole space $\hat A$ defined
by the Rota-Baxter operator $\widehat {\rm Id}$.
\end{rmk}

\subsection{
Commutative Connes cocycles} \ 

A {\bf Connes cocycle} on an associative algebra $(A,\cdot)$ is an
antisymmetric bilinear form $\mathcal{B}$ satisfying
\begin{align}\label{circulation}
\mathcal{B}(x\cdot y, z)+\mathcal{B}(y\cdot z,
x)+\mathcal{B}(z\cdot x, y)=0,\quad \forall x,y,z\in A.
\end{align}
It corresponds to the original definition of cyclic cohomology by
Connes (\cite{C}). Note that there is a close relation between
dendriform algebras and Connes cocycles (\cite{B}). Next we
consider the ``symmetric" version of Connes cocycle.

\begin{defi}
Let $(A,\cdot)$ be an associative algebra and $\mathcal{B}$ be a
bilinear form on $(A,\cdot)$. If $\mathcal{B}$ is symmetric and satisfies
Eq.~(\ref{circulation}), then $\mathcal{B}$ is called a {\bf
commutative Connes cocycle}.
\end{defi}

Let $(V, l, r)$ be a bimodule of an associative algebra
$(A,\cdot)$. Then $(V^*, r^*, l^*)$ is also a bimodule of
$(A,\cdot)$, where $V^*$ is the dual space of $V$ and the linear
maps $r^*,l^*:A\rightarrow {\rm End}_{\mathbb F}(V^*)$ are defined
respectively by
$$ \langle r^*(x)u^*, v\rangle=\langle u^*, r(x)v\rangle, \quad \langle l^*(x)u^*, v\rangle=\langle u^*, l(x)v\rangle,\quad \forall x\in A, u^*\in V^*, v\in
V.$$

\begin{thm}\label{compatible structure}
Let $(A,\cdot)$ be an associative algebra and $\mathcal{B}$ be a
nondegenerate  commutative Connes cocycle on $(A,\cdot)$. Then
there exists a compatible anti-dendriform algebra structure $(A$,
$\vartriangleright$, $\vartriangleleft)$ on $(A,\cdot)$ defined by
\begin{align}\label{non-comm}
\mathcal{B}(x\vartriangleright y, z)=-\mathcal{B}(y,z\cdot x),\ \
\mathcal{B}(x\vartriangleleft y, z)=-\mathcal{B}(x,y\cdot z),\ \
\forall x,y,z\in A.
\end{align}
\end{thm}

\begin{proof}
Define a linear map $T: A\rightarrow A^*$ by $$\langle T(x),
y\rangle=\mathcal{B}(x,y),\ \ \forall x,y\in A.$$ Then $T$ is
invertible since $\mathcal{B}$ is nondegenerate.
 For any $x,y,z\in A,$ we have
\begin{align*}
  \langle T(x\cdot y)+R.^*(x)T(y)+L.^*(y)T(x), z\rangle
 =&\langle T(x\cdot y),z\rangle+\langle T(y),R.(x)z\rangle+\langle T(x),L.(y)z\rangle \\
 =&\mathcal{B}(x\cdot y,z)+\mathcal{B}(z\cdot x,y)+\mathcal{B}(y\cdot z,x)=0,
\end{align*}
which implies that
$T$ is an anti-1-cocycle of $(A,\cdot)$ associated to $(A^*, R.^*, L.^*).$
So $T^{-1}: A^*\rightarrow A$ is an anti-$\mathcal{O}$-operator of $(A,\cdot)$ associated to $(A^*, R.^*, L.^*)$.

Note that for any $x,y\in A,$ there exist $a^*,b^*\in A^*$ such
that $x=T^{-1}(a^*), y=T^{-1}(b^*)$. By Theorem
\ref{anti-dendriform-invertible anti-O}, there is a compatible
anti-dendriform algebra structure on $(A,\cdot)$ defined by
\begin{align*}
x\vartriangleright y:&=T^{-1}(a^*\vartriangleright b^*)=-T^{-1}\big(R.^*(T^{-1}(a^*))b^*\big)=-T^{-1}\big(R.^*(x)T(y)\big),\\
x\vartriangleleft y:&=T^{-1}(a^*\vartriangleleft
b^*)=-T^{-1}\big(L.^*(T^{-1}(b^*))a^*\big)=-T^{-1}\big(L.^*(y)T(x)\big).
\end{align*}
Therefore, for any $x,y,z\in A$, we have
\begin{align*}
\mathcal{B}(x\vartriangleright y, z)&=
-\langle R.^*(x)T(y), z\rangle
                        =-\langle T(y),z\cdot x\rangle
                        =-\mathcal{B}(y, z\cdot x),\\
\mathcal{B}(x\vartriangleleft y, z)&=
-\langle L.^*(y)T(x), z\rangle
                        =-\langle T(x),y\cdot z\rangle
                        =-\mathcal{B}(x, y\cdot z).
\end{align*}
Thus the conclusion holds.
\end{proof}

\begin{cor}\label{commutative Connes cocycles on semi-direct algebras}
Let $(A,\vartriangleright,\vartriangleleft )$ be an anti-dendriform algebra and
$(A,\cdot)$ be the associated associative algebra.
Define a bilinear form $\mathcal{B}$ on $A\oplus A^{*}$ by
\begin{equation}\label{commutative  Connes cocycles on semi-direct algebras1}
\mathcal{B}(x+a^{*},y+b^{*})=\langle x,b^{*}\rangle+\langle
a^{*},y\rangle,\ \  \forall x,y\in A, a^{*},b^{*}\in A^{*}.
\end{equation}
Then $\mathcal{B}$ is a nondegenerate commutative Connes cocycle
on the associative algebra
 $A\ltimes_{-R^{*}_{\vartriangleleft},-L^{*}_{\vartriangleright}}A^*$.
Conversely, let $(A,\cdot)$ be an associative algebra and
$(A^*,l,r)$ be a bimodule of $(A,\cdot)$. Suppose that the
bilinear form given by Eq.~\eqref{commutative  Connes cocycles on
semi-direct algebras1} is a commutative Connes cocycle on
$A\ltimes_{l,r}A^*$. Then there is a compatible anti-dendriform
algebra structure $(A,\vartriangleright,\vartriangleleft )$ on
$(A,\cdot)$ such that  $l=-R^{*}_{\vartriangleleft},
r=-L^{*}_{\vartriangleright}$.
\end{cor}
\begin{proof}
\delete{ ($\Rightarrow$) It is easy to see that $\mathcal{B}$ is
nondegenerate and commutative. For any $x, y, z\in A, a^*, b^*,
c^*\in A^*,$ we compute
\begin{align*}
\mathcal{B}\bigg((x+a^*)(y+b^*), z+c^*\bigg)&=\mathcal{B}\bigg(x\cdot y-\mathcal{R}^{*}_{\vartriangleleft}(x)b^*
-\mathcal{L}^{*}_{\vartriangleright}(y)a^*, z+c^*\bigg)\\
&=\langle x\cdot y,c^*\rangle-\langle \mathcal{R}^{*}_{\vartriangleleft}(x)b^*, z\rangle-\langle \mathcal{L}^{*}_{\vartriangleright}(y)a^*, z\rangle \\
&=\langle x\cdot y,c^*\rangle-\langle b^*, z\vartriangleleft x \rangle-\langle a^*, y\vartriangleright z\rangle.
\end{align*}
Similarly, we have
\begin{align*}
\mathcal{B}\bigg((y+b^*)(z+c^*), x+a^*\bigg)
=\langle y\cdot z, a^*\rangle-\langle c^*, x\vartriangleleft y \rangle-\langle b^*, z\vartriangleright x\rangle,\\
\mathcal{B}\bigg((z+c^*)(x+a^*), y+b^*\bigg)
=\langle z\cdot x, b^*\rangle-\langle a^*, y\vartriangleleft z \rangle-\langle c^*, x\vartriangleright y\rangle.
\end{align*}
Thus $\mathcal{B}\bigg((x+a^*)(y+b^*),
z+c^*\bigg)+\mathcal{B}\bigg((y+b^*)(z+c^*),
x+a^*\bigg)+\mathcal{B}\bigg((z+c^*)(x+a^*), y+b^*\bigg)=0.$}

It is straightforward to show that $\mathcal{B}$ is a
nondegenerate commutative Connes cocycle on
$A\ltimes_{-R^{*}_{\vartriangleleft},-L^{*}_{\vartriangleright}}A^*$.
Conversely, by Theorem \ref{compatible structure}, there is a
compatible anti-dendriform algebra structure
$\vartriangleright,\vartriangleleft$ given by Eq.~\eqref{non-comm}
on $A\ltimes_{l,r}A^*.$ In particular, we have
$$\mathcal{B}(x\vartriangleright y, z)=-\mathcal{B}(y,z\cdot x)=0,\ \ \mathcal{B}(x\vartriangleleft y, z)=-\mathcal{B}(x,y\cdot z)=0,\ \ \forall x,y,z\in A.$$
Thus $x\vartriangleright y,x\vartriangleleft y\in A$ for all
$x,y\in A$ and hence $(A, \vartriangleright,\vartriangleleft)$ is
an anti-dendriform algebra. Furthermore, for all $x,y\in A, a^*\in
A^*,$ we have
\begin{align*}
&\langle -R_{\vartriangleleft }^*(x)a^*, y\rangle=-\langle a^*, y\vartriangleleft x\rangle=-\mathcal{B}( y\vartriangleleft x, a^*)
=\mathcal{B}( y, l(x) a^*)=\langle l(x) a^*, y\rangle,\\
&\langle -L_{\vartriangleright }^*(x)a^*, y\rangle=-\langle a^*, x\vartriangleright y\rangle=-\mathcal{B}( x\vartriangleright y, a^*)
=\mathcal{B}(y, r(x) a^*)=\langle r(x) a^*,y\rangle.
\end{align*}
So $l=-R^{*}_{\vartriangleleft},
r=-L^{*}_{\vartriangleright}$. 
\end{proof}

\begin{defi}\label{anti-dendriform invariant form}
Let $(A, \vartriangleright,\vartriangleleft)$ be an anti-dendriform algebra. A bilinear form $\mathcal{B}$ on $(A, \vartriangleright,\vartriangleleft)$ is called {\bf invariant} if
\begin{align*}
\mathcal{B}(x\vartriangleright y, z)=-\mathcal{B}(y,z\cdot x),\ \ \mathcal{B}(x\vartriangleleft y, z)=-\mathcal{B}(x,y\cdot z),\ \ \forall x,y,z\in A,
\end{align*}
where the bilinear operation $\cdot$ is defined by Eq.~(\ref{eq:asso}).
\end{defi}

The following conclusion is obvious.

\begin{lem}
Let $\mathcal{B}$ be an invariant bilinear form on an
anti-dendriform algebra $(A$, $\vartriangleright$,
$\vartriangleleft)$. Then $\mathcal{B}$ satisfies
$$\mathcal{B}(x\vartriangleleft y,z)=\mathcal{B}(z\vartriangleright x, y),\;\;\forall x,y,z\in A.$$
\end{lem}

\begin{pro}\label{symmetric invariant-commutative Connes}
Let $(A,\vartriangleright,\vartriangleleft)$ be an anti-dendriform
algebra and $\mathcal{B}$ be a symmetric invariant bilinear form
on $(A,\vartriangleright,\vartriangleleft).$ Then $\mathcal{B}$ is
a commutative Connes cocycle on the associated associative algebra
$(A,\cdot)$. Conversely, suppose that $(A,\cdot)$ is an associative
algebra and $\mathcal{B}$ is a nondegenerate commutative Connes
cocycle on $(A,\cdot).$ Then $\mathcal{B}$ is invariant on the
compatible anti-dendriform algebra
$(A,\vartriangleright,\vartriangleleft)$ defined by
Eq.~\eqref{non-comm}.
\end{pro}

\begin{proof}
For the first half part, for all $x,y,z\in A$, we have
\begin{eqnarray*}
  \mathcal{B}(x\cdot y,z)+\mathcal{B}(y\cdot z,x)+\mathcal{B}(z\cdot x,y)=\mathcal{B}(x\cdot
  y,z)-\mathcal B(x\vartriangleleft y, z)-\mathcal B(x\vartriangleright y, z)=0.
\end{eqnarray*}
So $\mathcal{B}$ is a commutative Connes cocycle on $(A,\cdot)$.
The second half part follows from Theorem \ref{compatible structure} immediately.
\end{proof}

Recall that two bimodules $(V_1,l_1,r_1)$ and $(V_2,l_2,r_2)$ of
an associative algebra $(A,\cdot)$ are called {\bf equivalent} if
there is a linear isomorphism $\varphi:V_1\rightarrow V_2$ such
that
$$\varphi(l_1(x)v_1)=l_2(x)\varphi(v_1),\ \ \varphi(r_1(x)v_1)=r_2(x)\varphi(v_1),\ \ \forall x\in A, v_1\in V_1.$$

\begin{pro}
Let $(A, \vartriangleright,\vartriangleleft)$ be an
anti-dendriform algebra. Then there is a nondegenerate invariant
bilinear form on $(A, \vartriangleright,\vartriangleleft)$ if and
only if $(A, -L_\vartriangleright,-R_\vartriangleleft)$ and $(A^*, R.^*, L.^*)$ are equivalent as bimodules of the associated
associative algebra $(A,\cdot)$.
\end{pro}
\begin{proof}
Suppose that $(A, -L_\vartriangleright,-R_\vartriangleleft)$ and $(A^*,
R.^*, L.^*)$ are equivalent as bimodules of $(A,\cdot)$. Then
there exists a linear isomorphism $\psi: A\rightarrow A^*$ such
that
$$\psi(-L_\vartriangleright(x)y)=R.^*(x)\psi(y),\ \ \psi(-R_\vartriangleleft(x)y)=L.^*(x)\psi(y),\ \ \forall x,y\in A.$$
Define a nondegenerate bilinear form $\mathcal{B}$ on $A$ as
\begin{align}\label{B-psi}
\mathcal{B}(x,y)=\langle \psi(x), y\rangle,\quad \forall x,y\in A.
\end{align}
For any $x,y,z\in A,$ we have
 \begin{align*}
   \mathcal{B}(x\vartriangleright y,z)&
   =-\langle\psi(-L_\vartriangleright(x)y),z\rangle
                          =-\langle R.^*(x)\psi(y), z\rangle
                          =-\mathcal{B}(y,z\cdot x),\\
   \mathcal{B}(x\vartriangleleft y,z)&
   =-\langle\psi(-R_\vartriangleleft(y)x),z\rangle
                          =-\langle L.^*(y)\psi(x),z\rangle
                          =-\mathcal{B}(x,y\cdot z).
 \end{align*}
So $\mathcal{B}$ is invariant.

Conversely, suppose that $\mathcal{B}$ is a nondegenerate invariant
bilinear form on $(A, \vartriangleright,\vartriangleleft).$ Define
a linear map $\psi:A\rightarrow A^*$ by Eq.~(\ref{B-psi}). By a
similar proof as above, we show that $\psi$ gives an equivalence
between $(A, -L_\vartriangleright,-R_\vartriangleleft)$ and $(A^*,
R.^*, L.^*)$ as bimodules of $(A,\cdot)$.
This completes the proof.
\end{proof}


Recall that a symmetric bilinear form $\mathcal{B}$ on a Lie
algebra $(A,[\,,\,])$ is called a {\bf commutative 2-cocycle} (see
\cite{DZ}) if the following equation holds:
$$\mathcal{B}([x,y],z)+\mathcal{B}([y,z],x)+\mathcal{B}([z,x],y)=0,\ \ \forall x,y,z\in A.$$
By \cite{LB}, there is a compatible anti-pre-Lie algebra structure
$(A,\circ)$ on a Lie algebra  $(A,[\,,\,])$ with a nondegenerate
commutative 2-cocycle $\mathcal B$ defined by
\begin{equation}\label{eq:antipl}\mathcal{B}(x\circ
y,z)=\mathcal{B}(y,[x,z]), \ \ \forall x,y,z\in A.\end{equation}
 A bilinear form $\mathcal{B}$ on an anti-pre-Lie algebra $(A,\circ)$ is called {\bf invariant} if Eq.~(\ref{eq:antipl}) holds.


\begin{lem} \label{lem:lem11}{\rm (\cite{LB})} Any symmetric invariant bilinear form on an anti-pre-Lie algebra
$(A,\circ)$ is a commutative 2-cocycle on the sub-adjacent Lie
algebra $(\frak g(A),[\,,\,])$. Conversely, a nondegenerate
commutative 2-cocycle on a Lie algebra $(A,[\,,\,])$ is invariant
on the compatible anti-pre-Lie algebra $(A,\circ)$ defined by Eq.~(\ref{eq:antipl}).
\end{lem}

\begin{lem} \label{lem:lem12} 
\begin{enumerate}
\item Let $(A,\vartriangleright,\vartriangleleft)$ be an
anti-dendriform algebra with a symmetric invariant bilinear form
$\mathcal{B}.$ Then $\mathcal{B}$ is  invariant on the associated
anti-pre-Lie algebra $(A,\circ).$ \label{symmetric
invariant-symmetric invariant}
\item Let $(A,\cdot)$ be an
associative algebra with a commutative Connes cocycle
$\mathcal{B}.$ Then $\mathcal{B}$ is a commutative 2-cocycle on
the sub-adjacent Lie algebra $(\frak g(A),[\, ,\,
]).$\label{commutative Connes-commutative 2-cocycle}
\end{enumerate}
\end{lem}

\begin{proof}
It is straightforward.
\end{proof}

\begin{pro}

Let $(A,\vartriangleright,\vartriangleleft)$ be an anti-dendriform
algebra with a symmetric invariant bilinear form $\mathcal B$.
Then the following conclusions hold:
\begin{enumerate}
\item $\mathcal{B}$ is a commutative Connes cocycle on the
associated associative algebra $(A,\cdot)$; \item  $\mathcal{B}$
is  invariant on the associated anti-pre-Lie algebra $(A,\circ)$;
\item $\mathcal{B}$ is  a commutative 2-cocycle on the
sub-adjacent Lie algebra $(\frak g(A),[\,,\,])$ of both
$(A,\cdot)$ and $(A,\circ)$.
\end{enumerate}
That is, the following diagram by ``putting" the symmetric
bilinear forms into the diagram~(\ref{diag:comm}) is commutative.
 \begin{equation} \label{eq:rbdiag1}
\begin{split}
{\tiny  \xymatrix{
   \txt{anti-dendriform algebra $(A,\vartriangleright,\vartriangleleft)$ with  \\ a symmetric invariant bilinear form $\mathcal{B}$}
    \ar@<1mm>@{->}[d]
    \ar@{->}[r]
    & \txt{anti-pre-Lie algebra $(A,\circ)$ with \\ a symmetric invariant bilinear form $\mathcal{B}.$}
     \ar@<1mm>@{->}[d]\\
\txt{associative algebra $(A,\cdot)$ with \\ a commutative Connes
cocycle $\mathcal{B}$}
   \ar@{->}[r]
      &\txt{ Lie algebra  $(\frak g(A),[\, ,\, ])$ with \\ a commutative 2-cocycle $\mathcal{B}$}
}}
\end{split}
\end{equation}
Conversely, let $(A,\cdot)$ be an associative algebra with a
nondegenerate commutative Connes cocycle $\mathcal{B}$. On the one
hand, $\mathcal B$ is a nondegenerate commutative 2-cocycle on the sub-adjacent
Lie algebra $(\frak g(A),[\, ,\, ])$ and hence there is a
compatible anti-pre-Lie algebra $(A,\circ)$ defined by
Eq.~(\ref{eq:antipl}) and $\mathcal B$ is invariant on
$(A,\circ)$. On the other hand, there is a compatible
anti-dendriform algebra $(A,\vartriangleright,\vartriangleleft)$
defined by Eq.~\eqref{non-comm} and $\mathcal B$ is invariant on
$(A,\vartriangleright,\vartriangleleft)$. Hence $\mathcal B$ is
invariant on the associated anti-pre-Lie algebra $(A,\circ')$
defined by Eq.~(\ref{eq:antipl1}). Therefore $(A,\circ)$ and
$(A,\circ')$ coincide, that is, the following diagram is
commutative.
 \begin{equation} \label{eq:rbdiag}
\begin{split}
{\tiny  \xymatrix{
   \txt{anti-dendriform algebra $(A,\vartriangleright,\vartriangleleft)$ with a \\ nondegenerate symmetric invariant bilinear form $\mathcal{B}$}
    \ar@<1mm>@{<-}[d]
    \ar@{->}[r]
    & \txt{anti-pre-Lie algebra $(A,\circ)$ with a \\ nondegenerate symmetric invariant bilinear form $\mathcal{B}.$}
     \ar@<1mm>@{<-}[d]\\
\txt{associative algebra $(A,\cdot)$ with a\\ nondegenerate
commutative Connes cocycle $\mathcal{B}$}
   \ar@{->}[r]
      &\txt{ Lie algebra  $(\frak g(A),[\, ,\, ])$ with a\\ nondegenerate commutative 2-cocycle $\mathcal{B}$}
}}
\end{split}
\end{equation}

\end{pro}

\begin{proof}
By the first half parts of Proposition~\ref{symmetric
invariant-commutative Connes} and Lemma~\ref{lem:lem11}
respectively, and Lemma~\ref{lem:lem12}, the first half part
follows. For the second half part, note that for any $x,y,z\in A$,
we have
\begin{eqnarray*}
\mathcal B(x\circ y, z)=\mathcal B(y,[x,z])=\mathcal B(y,x\cdot
z-z\cdot x)=\mathcal B(x\vartriangleright y, z)-\mathcal
B(y\vartriangleleft x,z)=\mathcal B(x\circ' y,z).
\end{eqnarray*}
Hence $x\circ y=x\circ'y$. Then the conclusion follows immediately
from the second half parts of Proposition~\ref{symmetric
invariant-commutative Connes} and Lemma~\ref{lem:lem11}
respectively, and Lemma~\ref{lem:lem12}.
\end{proof}

\section{Correspondences between some subclasses of dendriform and anti-dendriform algebras }

We give the correspondence between some subclasses of
dendriform algebras and anti-dendriform algebras in terms of
$q$-algebras. We also generalize the correspondence between some
subclasses of pre-Lie algebras and anti-pre-Lie algebras from $q=-
2$ in \cite{LB} to any $q\ne 0,\pm1$ and hence the relationships
between dendriform algebras and the associated pre-Lie algebras as
well as anti-dendriform algebras and the associated anti-pre-Lie
algebras are still kept on these subclasses for a fixed $q$. Therefore in the case that $q=- 2$, the notions of Novikov-type dendriform algebras and admissible Novikov-type
dendriform algebras are introduced as analogues of Novikov algebras and admissible Novikov algebras for
dendriform algebras and anti-dendriform algebras respectively.

Throughout this section, $q\in \mathbb{F}$ satisfying $q\ne 0,
\pm1.$

\subsection{Correspondences between some subclasses of dendriform and anti-dendriform algebras}\label{Correspondence-1}





\begin{defi}
Let $A$ be a vector space with two bilinear operations
$\succ,\prec.$ Define two bilinear operations
$\vartriangleright,\vartriangleleft:A\otimes A\rightarrow A$
respectively by
\begin{equation}\label{eq:q-al}
 x\vartriangleright y=x\succ y+qx\prec y, \quad x\vartriangleleft y=x\prec  y+qx\succ y,\quad  \forall x,y\in A.
\end{equation}
Then the triple $(A,\vartriangleright,\vartriangleleft)$ is called a $q$-{\bf algebra} of $(A,\succ,\prec).$
\end{defi}

\begin{rmk} There is an alternative choice of $q$-algebras for the
triple $(A,\succ,\prec)$. Let $A$ be a vector space with two
bilinear operations $\succ,\prec.$ Define two bilinear operations
$\vartriangleright',\vartriangleleft':A\otimes A\rightarrow A$
respectively by
\begin{eqnarray}\label{another q-algebra}
 x\vartriangleright' y=x\succ y+qy\succ x, \quad x\vartriangleleft' y=x\prec  y+qy\prec x,\quad  \forall x,y\in A.
\end{eqnarray}
However,  such an approach is not ``naturally available" for
associative admissible algebras such as dendriform as well as
anti-dendriform algebras. In fact, suppose that $(A,\succ,\prec)$ is an
associative admissible algebra. Then we have the following conclusions.
\begin{enumerate}
\item By Eq.~\eqref{eq:q-al},
$(A,\vartriangleright,\vartriangleleft)$ is always an associative
admissible algebra. \item   If $q\ne 0$, then from
Eq.~\eqref{another q-algebra},
$(A,\vartriangleright',\vartriangleleft')$ is an associative
admissible algebra if and only if the $q$-algebra (see
Definition~\ref{algebra-q-algebra}) of the associated associative
algebra $(A,\cdot)$ of $(A,\succ,\prec)$, where $\cdot$ is defined
by Eq.~(\ref{dendriform algebra'}), is still an associative algebra. Note that
the latter holds if and only if the sub-adjacent Lie algebra
$(\frak g(A),[\,,\,])$ of $(A,\cdot)$ is two-step nilpotent, that
is, $[[x,y],z]=0$ for all $x,y,z\in A$.
 \end{enumerate}
 Hence in the sense of keeping the property of splitting the
 associativity for both an associative admissible algebra  $(A,\succ,\prec)$ and its
 $q$-algebra, it is natural to use Eq.~\eqref{eq:q-al} (not Eq.~\eqref{another
 q-algebra}) to define the $q$-algebra of the associative admissible algebra $(A,\succ,\prec)$.
\end{rmk}

\begin{rmk}\label{rem.3.2}
When $q=0$, the $0$-algebra of $(A,\succ,\prec)$ is itself.
Moreover, note that
$$x\vartriangleright y-qx\vartriangleleft y=(1-q^2)x\succ y,\ \ x\vartriangleleft y-qx\vartriangleright y=(1-q^2)x\prec y,\ \ \forall x,y\in A.$$
Hence we have the following conclusions.
\begin{enumerate}
\item When $q\ne \pm1$,  the bilinear operations $\succ,\prec$ can
be presented by $\vartriangleright,\vartriangleleft$. Furthermore,
the $-q$-algebra of $(A,\vartriangleright,\vartriangleleft)$ has
the same algebra structure as $(A,\succ,\prec)$. \item When
$q=\pm1$,  the bilinear operations $\succ,\prec$ cannot be
presented by $\vartriangleright,\vartriangleleft$. Furthermore,
the $-q$-algebra of $(A,\vartriangleright,\vartriangleleft)$ is
trivial.
\end{enumerate}
So in the sense that the triple $(A,\succ,\prec)$ and its
$q$-algebra can be non-trivially presented by each other, the cases
that $q=0,\pm1$ are excluded.
\end{rmk}

\begin{pro}\label{dendriform-anti-dendriform}
Let $(A,\succ,\prec)$ be a dendriform algebra. Denote by $(A,
\vartriangleright,\vartriangleleft)$ the $q$-algebra of
$(A,\succ,\prec).$ Then $(A,\vartriangleright,\vartriangleleft)$
is an anti-dendriform algebra if and only if $(A,\succ,\prec)$
satisfies the following equations:
\begin{equation}\label{eq:s1}
x\succ(y\succ z)=(x\prec y)\prec z,
\end{equation}
\begin{equation}\label{eq:s2}
(x\prec y)\succ z=x\prec(y\succ z),
\end{equation}
\begin{equation}\label{dendri-anti-dendri}
(q^2+3q+2)(x\prec y)\prec z+(q^2+2q)x\succ(y\prec
z)+(q^2-q)x\prec(y\prec z)=0,
\end{equation}
for all $ x,y,z\in A$.
\end{pro}

\begin{proof}
Let $x,y,z\in A$.  By Eq.~(\ref{eq:q-al}) and the definition of
dendriform algebras, we have
{\small \begin{eqnarray}
&&x\vartriangleright(y\vartriangleright z)+(x\vartriangleright y)\vartriangleright z+(x\vartriangleleft y)\vartriangleright z\notag\\
&&=2x\succ(y\succ z)+q\big(2(x\succ y)\prec z+x\prec (y\succ z)+(x\prec y)\succ z+(x\succ y)\succ z+(x\prec y)\prec z\big)\notag\\
&&\hspace{0.4cm} +q^2\big(x\prec (y\prec z)+(x\prec y)\prec z+(x\succ y)\prec z\big),\label{anti-dendriform-2}\\
&&(x\vartriangleleft y)\vartriangleleft z+ x\vartriangleleft (y\vartriangleright z)+ x\vartriangleleft (y\vartriangleleft z)\notag\\
&&=2(x\prec y)\prec z+q\big(2x\succ (y\prec z)+(x\prec y)\succ z+x\prec (y\prec z)+x\succ (y\succ z)+x\prec (y\succ z)\big)\notag\\
&&\hspace{0.4cm}+q^2\big((x\succ y)\succ z+x\succ (y\prec z)+x\succ (y\succ z)\big),\label{anti-dendriform-3}\\
&&x\vartriangleright(y\vartriangleright z)-(x\vartriangleleft y)\vartriangleleft z\notag\\
&&=x\succ(y\succ z)-(x\prec y)\prec z+q\big(x\prec (y\succ z)-(x\prec y)\succ z\big)+q^2\big(x\prec (y\prec z)-(x\succ y)\succ z\big),\label{anti-dendriform-4}\\
&&(x\vartriangleright y)\vartriangleleft z-x\vartriangleright (y\vartriangleleft z)\notag\\
&&=q\big((x\prec y)\prec z+(x\succ y)\succ z-x\succ (y\succ
z)-x\prec (y\prec z)\big)+q^2\big((x\prec y)\succ z-x\prec (y\succ
z)\big).\label{anti-dendriform-5}
\end{eqnarray}
} Therefore $(A,\vartriangleright,\vartriangleleft)$ is an
anti-dendriform algebra if and only if the right hand sides of
Eqs.~(\ref{anti-dendriform-2})-(\ref{anti-dendriform-5}) are zero.
Next we assume that the right hand sides of
Eqs.~(\ref{anti-dendriform-2})-(\ref{anti-dendriform-5}) are zero
and we still denote them by
Eqs.~(\ref{anti-dendriform-2})-(\ref{anti-dendriform-5})
respectively. Thus we have the following interpretation.
\begin{enumerate}
\item The difference between Eq.~\eqref{anti-dendriform-2} and
Eq.~\eqref{anti-dendriform-3} is
\begin{eqnarray}\label{anti-dendriform-6}
&&2x\succ(y\succ z)-2(x\prec y)\prec z+q\big((x\succ y)\succ z+(x\prec y)\prec z-x\prec (y\prec z)-x\succ (y\succ z)\big)\notag\\
&&\hspace{0.4cm}+q^2\big(x\prec (y\prec z)+(x\prec y)\prec z-(x\succ y)\succ z-x\succ (y\succ z)\big)\\
&&=(2-q-q^2)\big(x\succ(y\succ z)-(x\prec y)\prec
z\big)+(q-q^2)\big((x\succ y)\succ z-x\prec (y\prec
z)\big)=0.\notag
\end{eqnarray}
\item The difference between Eq.~\eqref{anti-dendriform-6} and
Eq.~\eqref{anti-dendriform-5} is
\begin{align}\label{anti-dendriform-7}
&(2-2q^2)\big(x\succ(y\succ z)-(x\prec y)\prec z\big)=0.
\end{align}
By the assumption of $q$, Eq.~(\ref{anti-dendriform-7}) holds if
and only if Eq.~(\ref{eq:s1}) holds.

\item Suppose that Eqs.~\eqref{anti-dendriform-6} and
\eqref{eq:s1} hold. Then Eq.~\eqref{anti-dendriform-4} holds if
and only if the following equation holds:
\begin{equation*}\label{anti-dendriform-8}
x\prec (y \succ z)-(x\prec y)\succ z=0,
\end{equation*}
that is, Eq.~(\ref{eq:s2}) holds.

\item Suppose that Eqs.~\eqref{eq:s1}
and \eqref{eq:s2} hold. Then by the
definition of dendriform algebras, we have
\begin{enumerate}
\item Eq.~\eqref{anti-dendriform-6} holds;
\item Eq.~\eqref{anti-dendriform-3} holds
if and only if the following equation holds:
\begin{align*}
(q^2+3q+2)(x\prec y)\prec z+(q^2+2q)x\succ (y\prec
z)+(q^2-q)x\prec (y\prec z)=0,
\end{align*}
that is, Eq.~(\ref{dendri-anti-dendri}) holds.
\end{enumerate}
\end{enumerate}
Therefore $(A,\vartriangleright,\vartriangleleft)$ is an
anti-dendriform algebra if and only if the following equivalences
hold:
\begin{eqnarray*}
{\rm
Eqs.}~\eqref{anti-dendriform-2},~\eqref{anti-dendriform-3},~\eqref{anti-dendriform-4}
{\rm \ and\ }\eqref{anti-dendriform-5} {\rm \ hold}.
&\Longleftrightarrow\ &{\rm Eqs.}~\eqref{eq:s1},~\eqref{anti-dendriform-3},~\eqref{anti-dendriform-4}{\rm \ and\ }~\eqref{anti-dendriform-6}{\rm \ hold}.\\
&\Longleftrightarrow\ &{\rm Eqs.}~\eqref{eq:s1},~\eqref{eq:s2}{\rm
\ and\ }  \eqref{dendri-anti-dendri}{\rm \ hold}.
\end{eqnarray*}
Therefore the conclusion holds.
\end{proof}

\begin{pro}\label{anti-dendriform-dendriform}
Suppose that $(A,\vartriangleright,\vartriangleleft)$ is an
anti-dendriform algebra. Denote by $(A,\succ,\prec)$ the
$-q$-algebra of $(A, \vartriangleright,\vartriangleleft).$ Then
$(A,\succ,\prec)$ is a dendriform algebra if and only if $(A,
\vartriangleright,\vartriangleleft)$ satisfies the following
equations:
\begin{equation}\label{eq:anti-s1}
(x\vartriangleleft y)\vartriangleright z=x\vartriangleleft
(y\vartriangleright z),
\end{equation}
\begin{equation}\label{anti-dendri-dendri}
(-q^2+q+2)(x\vartriangleleft y)\vartriangleleft
z-q^2(x\vartriangleright y)\vartriangleleft
z+(q^2+q)x\vartriangleleft (y\vartriangleleft z)=0,
\end{equation}
for all $ x,y,z\in A$.
\end{pro}

\begin{proof}
\delete{By definition of the $-q$-algebra, for any $x,y,z\in A,$ we firstly get
\begin{align*}
x\succ(y\succ z)&=x\vartriangleright (y\vartriangleright z)-qx\vartriangleright (y\vartriangleleft z)-qx\vartriangleleft (y\vartriangleright z)+q^2x\vartriangleleft (y\vartriangleleft z),\notag\\
(x\succ y)\succ z&=(x\vartriangleright y)\vartriangleright z-q(x\vartriangleleft y)\vartriangleright z-q(x\vartriangleright y)\vartriangleleft z+q^2(x\vartriangleleft y)\vartriangleleft z,\notag\\
(x\prec y)\succ z&=(x\vartriangleleft y)\vartriangleright z-q(x\vartriangleright y)\vartriangleright z -q(x\vartriangleleft y)\vartriangleleft z+q^2(x\vartriangleright y)\vartriangleleft z,\notag\\
(x\prec y)\prec z&=(x\vartriangleleft y)\vartriangleleft z-q(x\vartriangleright y)\vartriangleleft z-q(x\vartriangleleft y)\vartriangleright z+q^2(x\vartriangleright y)\vartriangleright z,\notag\\
x\prec(y\prec z)&=x\vartriangleleft (y\vartriangleleft z)-qx\vartriangleleft (y\vartriangleright z)-qx\vartriangleright (y\vartriangleleft z)+q^2x\vartriangleright (y\vartriangleright z),\notag\\
x\prec(y\succ z)&=x\vartriangleleft (y\vartriangleright z)-qx\vartriangleleft (y\vartriangleleft z)-qx\vartriangleright (y\vartriangleright z)+q^2x\vartriangleright (y\vartriangleleft z),\notag\\
(x\succ y)\prec z&=(x\vartriangleright y)\vartriangleleft z-q(x\vartriangleleft y)\vartriangleleft z-q(x\vartriangleright y)\vartriangleright z+q^2(x\vartriangleleft y)\vartriangleright z,\notag\\
x\succ(y\prec z)&=x\vartriangleright (y\vartriangleleft z)-qx\vartriangleright (y\vartriangleright z)-qx\vartriangleleft (y\vartriangleleft z)+q^2x\vartriangleleft (y\vartriangleright z),\notag
\end{align*}
which imply} Let $x,y,z\in A$.  By the  definitions of $q$-algebras
and anti-dendriform algebras, we have {\small \begin{eqnarray}
&&x\succ(y\succ z)-(x\succ y)\succ z-(x\prec y)\succ z\notag\\
&&=2x\vartriangleright (y\vartriangleright z)+q\big(x\vartriangleleft (y\vartriangleright z)+(x\vartriangleright y)\vartriangleright z+(x\vartriangleleft y)\vartriangleleft z-(x\vartriangleleft y)\vartriangleright z\big)\notag\\
&&\hspace{0.4cm} +q^2\big(x\vartriangleleft (y\vartriangleleft z)-(x\vartriangleleft y)\vartriangleleft z-(x\vartriangleright y)\vartriangleleft z\big),\label{dendriform algebra-1}\\
&&(x\prec y)\prec z-x\prec(y\prec z)-x\prec(y\succ z)\notag\\
&&=2(x\vartriangleleft y)\vartriangleleft z+q\big(x\vartriangleleft (y\vartriangleright z)+x\vartriangleleft (y\vartriangleleft z)+x\vartriangleright (y\vartriangleright z)-(x\vartriangleleft y)\vartriangleright z\big)\notag\\
&&\hspace{0.4cm} +q^2\big((x\vartriangleright y)\vartriangleright z-x\vartriangleright (y\vartriangleright z)-x\vartriangleright (y\vartriangleleft z)\big),\label{dendriform algebra-2}\\
&&(x\succ y)\prec z-x\succ(y\prec z)=(q^2+q)\big((x\vartriangleleft y)\vartriangleright z-x\vartriangleleft (y\vartriangleright z)\big).\label{dendriform algebra-3}
\end{eqnarray}}
So $(A,\succ,\prec)$ is a dendriform algebra if and only if the right hand sides of Eq.~(\ref{dendriform algebra-1})-(\ref{dendriform algebra-3}) are zero.
Now we assume that the right hand sides of
Eqs.~(\ref{dendriform algebra-1})-(\ref{dendriform algebra-3}) are zero
and we still denote them by
Eqs.~(\ref{dendriform algebra-1})-(\ref{dendriform algebra-3})
respectively. Thus we have the following interpretation.
\delete{\begin{align*}
&x\succ(y\succ z)-(x\succ y)\succ z-(x\prec y)\succ z=0,\\
&(x\prec y)\prec z-x\prec(y\prec z)-x\prec(y\succ z)=0,\\
&(x\succ y)\prec z-x\succ(y\prec z)=0,
\end{align*}
which are equivalent to
\begin{align}
&2x\vartriangleright (y\vartriangleright z)+q\big(x\vartriangleleft (y\vartriangleright z)+(x\vartriangleright y)\vartriangleright z+(x\vartriangleleft y)\vartriangleleft z-(x\vartriangleleft y)\vartriangleright z\big)\notag\\
&+q^2\big(x\vartriangleleft (y\vartriangleleft z)-(x\vartriangleleft y)\vartriangleleft z-(x\vartriangleright y)\vartriangleleft z\big)=0,\label{dendriform-2}\\
&2(x\vartriangleleft y)\vartriangleleft z+q\big(x\vartriangleleft (y\vartriangleright z)+x\vartriangleleft (y\vartriangleleft z)+x\vartriangleright (y\vartriangleright z)-(x\vartriangleleft y)\vartriangleright z\big)\notag\\
&+q^2\big((x\vartriangleright y)\vartriangleright z-x\vartriangleright (y\vartriangleright z)-x\vartriangleright (y\vartriangleleft z)\big)=0,\label{dendriform-3}\\
&q\big(x\vartriangleleft (y\vartriangleleft z)-(x\vartriangleright y)\vartriangleright z\big)+q^2\big((x\vartriangleleft y)\vartriangleright z-x\vartriangleleft (y\vartriangleright z)\big)=0.\label{dendriform-4}
\end{align}}
\begin{enumerate}
\item By the assumption of $q$, Eq.~(\ref{dendriform algebra-3})
holds if and only if Eq.~(\ref{eq:anti-s1}) holds. \item By the
definition of anti-dendriform algebras, the difference between
Eq.~(\ref{dendriform algebra-1}) and Eq.~(\ref{dendriform
algebra-2}) is Eq.~(\ref{dendriform algebra-3}). Therefore after
supposing that Eq.~(\ref{eq:anti-s1}) holds,  we show that
Eq.~(\ref{dendriform algebra-1}) holds if and only if
Eq.~(\ref{dendriform algebra-2}) holds. \item  Suppose that
Eq.~(\ref{eq:anti-s1}) holds. By the definition of anti-dendriform
algebras again, Eq.~(\ref{dendriform algebra-1}) holds if and only
if the following equation holds:
$$
(-q^2+q+2)(x\vartriangleleft y)\vartriangleleft z+(q^2+q)x\vartriangleleft (y\vartriangleleft z)-q^2(x\vartriangleright y)\vartriangleleft z=0,$$
that is, Eq.~(\ref{anti-dendri-dendri}) holds.
\end{enumerate}
Hence $(A,\succ,\prec)$ is a dendriform algebra if and only if the
following equivalences hold:
\begin{eqnarray*}
{\rm
Eqs.}~\eqref{dendriform algebra-1},~\eqref{dendriform algebra-2},{\rm \ and\ }\eqref{dendriform algebra-3}{\rm \ hold}.
&\Longleftrightarrow\ &{\rm Eqs.}~\eqref{eq:anti-s1}{\rm \ and\ } ~\eqref{dendriform algebra-1}{\rm \ hold}.\\
&\Longleftrightarrow\ &{\rm Eqs.}~\eqref{eq:anti-s1}{\rm \ and\ } ~\eqref{anti-dendri-dendri}{\rm \ hold}.
\end{eqnarray*}
This completes the proof.
\end{proof}



\begin{thm}\label{main-1}
Let $A$ be a vector space with two bilinear operations
$\succ,\prec$. Then $(A,\succ,\prec)$ is a dendriform algebra
satisfying Eqs.~\eqref{eq:s1}-\eqref{dendri-anti-dendri} if and
only if its $q$-algebra $(A,\vartriangleright,\vartriangleleft)$
is an anti-dendriform algebra  satisfying
Eqs.~\eqref{eq:anti-s1}-\eqref{anti-dendri-dendri}.
\end{thm}
\begin{proof}
Suppose that $(A,\succ,\prec)$ is a dendriform algebra satisfying
Eqs.~(\ref{eq:s1})-(\ref{dendri-anti-dendri}). Then it is clear
that $(A,\vartriangleright,\vartriangleleft)$ is an
anti-dendriform algebra by Proposition
\ref{dendriform-anti-dendriform}. Furthermore, note that
$-q$-algebra of $(A,\vartriangleright,\vartriangleleft)$ is a
dendriform algebra, thus
Eqs.~(\ref{eq:anti-s1})-(\ref{anti-dendri-dendri}) hold by
Proposition \ref{anti-dendriform-dendriform}, that is,
$(A,\vartriangleright,\vartriangleleft)$ is an anti-dendriform
algebra satisfying
Eqs.~(\ref{eq:anti-s1})-(\ref{anti-dendri-dendri}).
 The converse is
similar.
\end{proof}
\begin{rmk}
Theorem \ref{main-1} is equivalent to the following statement. The
triple $(A,\vartriangleright,\vartriangleleft)$ is an
anti-dendriform algebra satisfying
Eqs.~\eqref{eq:anti-s1}-\eqref{anti-dendri-dendri} if and only if
its $(-q)$-algebra $(A,\succ,\prec)$ is a dendriform algebra
satisfying Eqs.~\eqref{eq:s1}-\eqref{dendri-anti-dendri}.
\end{rmk}

Obviously, for Eq.~(\ref{dendri-anti-dendri}), $q=-2$ is a little
``special" in the sense that only one monomial in $x,y,z$ is left, giving the
following notion.

\begin{defi} Let $(A,\succ,\prec)$ be a dendriform algebra.
Then $(A,\succ,\prec)$ is called a {\bf
Novikov-type dendriform algebra} if
Eqs.~(\ref{eq:s1})-(\ref{eq:s2}) and the following equation hold:
\begin{equation}\label{eq:NTD}x\prec(y\prec z)=0,\quad \forall
x,y,z\in A.\end{equation}
\end{defi}

\begin{pro}\label{pro:equiv}
Let $A$ be a vector space with two bilinear operations $\succ,\prec$. Then
$(A,\succ,\prec)$ is a Novikov-type dendriform algebra if and only if the following equations hold:
\begin{eqnarray}
&&x\succ(y\succ z)=(x\prec y)\prec z=x\prec(y\succ z)=(x\prec y)\succ z,\label{eq:ND1}\\
&&x\succ(y\prec z)=(x\succ y)\prec z,\label{eq:ND2}\\
&&(x\succ y)\succ z=x\prec(y\prec z)=0,\label{eq:ND3},
\end{eqnarray}
for all $x,y,z\in A$.
\end{pro}

\begin{proof}
Let $x,y,z\in A$. Then we set all products involving $x,y,z$ as
variables, that is, there are the following 8 variables {\small
$$(x\succ y)\succ z, x\succ (y\succ z), (x\prec y)\prec z, x\prec
(y\prec z), (x \succ y)\prec z, x\succ (y\prec z), (x\prec y)\succ
z, x\prec (y\succ z).$$} Therefore, Eqs.~(\ref{dendriform
algebra}), (\ref{eq:s1}), (\ref{eq:s2}) and (\ref{eq:NTD}) compose
a set of linear equations in these variables. It is
straightforward to show that the solution of these linear
equations is given by Eqs.~(\ref{eq:ND1})-(\ref{eq:ND3}) with the
two free variables $(x\prec y)\succ z$ and $(x\succ y)\prec z$,
that is, the other variables are the linear combinations of
$(x\prec y)\succ z$ and $(x\succ y)\prec z$. Thus the conclusion
holds.
\end{proof}

For the corresponding case of anti-dendriform algebras, we give the following notion.

\begin{defi} Let
$(A,\vartriangleright,\vartriangleleft)$ be an anti-dendriform algebra. Then
$(A,\vartriangleright,\vartriangleleft)$ is called an {\bf
admissible Novikov-type dendriform algebra} if
Eq.~(\ref{eq:anti-s1}) and the following equation hold:
\begin{equation}\label{eq:NTAD} x\vartriangleleft (y\vartriangleleft z)=2(x\cdot
y)\vartriangleleft z,\;\;\forall x,y,z\in A,\end{equation} where
the bilinear operation $\cdot$ is defined by Eq.~(\ref{eq:asso}),
that is, $x\cdot y=x\vartriangleright y+x\vartriangleleft y$ for
all $x,y\in A$.
\end{defi}

\begin{pro}
Let $A$ be a vector space with two bilinear operations $\vartriangleright,\vartriangleleft$.
Then $(A,\vartriangleright,\vartriangleleft)$ is an admissible Novikov-type dendriform algebra
if and only if the following equations hold:
\begin{eqnarray}
&&(x\vartriangleright y)\vartriangleright z=x\vartriangleleft (y\vartriangleleft z)
        =\frac{2}{3}(x\vartriangleright y)\vartriangleleft z-\frac{2}{3}(x\vartriangleleft y)\vartriangleright z,\label{eq:AND-1}\\
        &&x\vartriangleright (y\vartriangleright z)=(x\vartriangleleft y)\vartriangleleft z
        =-\frac{2}{3}(x\vartriangleright y)\vartriangleleft z-\frac{1}{3}(x\vartriangleleft y)\vartriangleright z,\label{eq:AND2}\\
        &&x\vartriangleright (y\vartriangleleft z)=(x\vartriangleright y)\vartriangleleft z,\label{eq:AND3}\\
 && x\vartriangleleft (y\vartriangleright z)=(x\vartriangleleft y)\vartriangleright z,\label{eq:AND4}
\end{eqnarray}
for all $x,y,z\in A$.
\end{pro}

\begin{proof}
It is similar to the one for Proposition~\ref{pro:equiv}.
\end{proof}

\begin{cor}\label{Corollary q=-2}
     Let $A$ be a vector space with two bilinear operations
$\succ,\prec$. The triple $(A,\succ,\prec)$ is a Novikov-type
dendriform algebra if and only if its $-2$-algebra
$(A,\vartriangleright,\vartriangleleft)$ is an admissible
Novikov-type dendriform algebra.
\end{cor}
\begin{proof}
Note that when $q=-2$, Eq.~(\ref{eq:NTD}) holds if and only if
Eq.~\eqref{dendri-anti-dendri}  holds, and Eq.~(\ref{eq:NTAD}) holds
if and only if  Eq.~\eqref{anti-dendri-dendri} holds. Hence the
conclusion follows from Theorem~\ref{main-1}.
\end{proof}

\begin{ex}
It is obvious that all ``2-nilpotent" dendriform algebras in the sense that all products involving three elements are zero (see Example~\ref{ex:2-dim}) are Novikov-type
dendriform algebras. In particular, any 2-nilpotent associative algebra $(A,\cdot)$ gives a Novikov-type dendriform algebra $(A,\succ, \prec)$ by letting
$\succ=\cdot, \prec=0$ or $\succ=0,\prec=\cdot$. Accordingly, all ``2-nilpotent" anti-dendriform algebras are admissible Novikov-type dendriform algebras.
In particular,  all complex anti-dendriform algebras in dimensions 1 and 2 which are classified in Examples~\ref{one-dim} and \ref{ex:2-dim} respectively are admissible Novikov-type dendriform algebras. Note that the 3-dimensional anti-dendriform algebras given in Example~\ref{3-dimension} are not admissible Novikov-type dendriform algebras.
\end{ex}


\subsection{More correspondences and their relationships 
}\label{Correspondence-2}

\begin{defi}\label{algebra-q-algebra}
Let $A$ be a vector space with a bilinear operation $\bullet.$
Define a bilinear operation $\diamond$ as
\begin{eqnarray}\label{one-q-algebra}
x\diamond y=x\bullet y+q y\bullet x,\quad \forall x,y\in A.
\end{eqnarray}
Then $(A,\diamond)$ is called a {\bf $q$-algebra} of
$(A,\bullet).$
\end{defi}

Recall that a {\bf pre-Lie algebra} is a vector space $A$ with a
bilinear operation $\ast$ satisfying \begin{equation}(x\ast y)\ast
z-x\ast(y\ast z)=(y\ast x)\ast z-y\ast(x\ast z),\quad \forall
x,y,z\in A.\end{equation}

A \textbf{Novikov algebra} (\cite{Bal,Gel}) is a pre-Lie algebra
$(A,\ast)$ such that
\begin{equation}\label{eq:defi:Novikov algebras1}
(x\ast y)\ast z=(x\ast z)\ast y, \quad\forall x,y,z\in A.
\end{equation}

An \textbf{admissible Novikov algebra} is a vector space with a
bilinear operation $\circ$ satisfying Eq.~(\ref{eq:21}) and the
following equation:
\begin{equation}\label{eq:defi:admissible Novikov algebras1}
2x\circ[y,z]=(x\circ y)\circ z-(x\circ z)\circ y, \;\; \forall
x,y,z\in A.
\end{equation}
It is known that an admissible Novikov algebra is an anti-pre-Lie
algebra (\cite{LB}).

\begin{pro}\label{pre-Lie-anti-pre-Lie}
Let $(A,\ast)$ be a pre-Lie algebra. Denote by $(A,\circ)$ the
$(-q)$-algebra of $(A,\ast).$ Then $(A,\circ)$ is an anti-pre-Lie
algebra if and only if the following equation holds:
\begin{align}\label{pre-Lie-anti}
(2+q)[x,y]\ast z+(-q^2-2q)z\ast[x,y]+(q^2-q)((z\ast y)\ast x-(z\ast x)\ast y)=0,\ \ \forall x,y,z\in A,
\end{align}
where $[x,y]=x\ast y-y\ast x$.
\end{pro}


\begin{proof}
Let $x,y,z\in A$. By Eq.~(\ref{one-q-algebra}), we have
$$[x,y]_\circ=x\circ y-y\circ x=x\ast y-qy\ast x-y\ast x+qx\ast y=(1+q)[x,y].$$
So $(A,\circ)$ is a Lie-admissible algebra. Furthermore, by Eq.~(\ref{one-q-algebra}) and the definition of pre-Lie algebras,
we have
\begin{eqnarray}
&&x\circ(y\circ z)-y\circ(x\circ z)-[y,x]_\circ\circ z\notag\\
&&=(2+q)[x,y]\ast z+(-q^2-2q)z\ast[x,y]+(q^2-q)\big((z\ast y)\ast x-(z\ast x)\ast y\big).\label{pre-Lie-anti-1}
\end{eqnarray}
Therefore $(A,\circ)$ is an anti-pre-Lie algebra if and only if
the right hand side of Eq.~(\ref{pre-Lie-anti-1}) is zero. Hence
the conclusion follows.
\end{proof}

\begin{rmk}
Note that when $q=-2$, Eq.~(\ref{pre-Lie-anti}) holds if and only
if Eq.~(\ref{eq:defi:Novikov algebras1}) holds, that is, in this
case, a pre-Lie algebra satisfying Eq.~(\ref{pre-Lie-anti}) is
exactly a Novikov algebra.
\end{rmk}

\begin{pro}\label{anti-pre-Lie-pre-Lie}
Let $(A,\circ)$  be an anti-pre-Lie algebra. Denote by $(A,\ast)$ the $q$-algebra of $(A,\circ).$
Then $(A,\ast)$ is a pre-Lie algebra if and only if the following equation holds:
\begin{align}\label{anti-pre-Lie-pre}
(2+q)[x,y]_\circ\circ z-q^2z\circ[x,y]_\circ+(q^2+q)\big((z\circ x)\circ y-(z\circ y)\circ x\big)=0,\ \ \forall x,y,z\in A,
\end{align}
where $[x,y]_\circ=x\circ y-y\circ x$.
\end{pro}
\begin{proof}
Let $x,y,z\in A$. By Eq.~(\ref{one-q-algebra}) and the definition
of anti-pre-Lie algebras, we have
\begin{eqnarray}
&&(x\ast y)\ast z-x\ast(y\ast z)-(y\ast x)\ast z+y\ast(x\ast z)\notag\\
&&=(2+q)[x,y]_\circ\circ z-q^2z\circ[x,y]_\circ+(q^2+q)((z\circ x)\circ y-(z\circ y)\circ x).\label{anti-pre-Lie-pre-Lie-1}
\end{eqnarray}
Therefore $(A,\ast)$ is a pre-Lie algebra if and only if the right
hand side of Eq.~(\ref{anti-pre-Lie-pre-Lie-1}) is zero. This
completes the proof.
\end{proof}

\begin{rmk}
Note that when $q=-2$, Eq.~(\ref{anti-pre-Lie-pre}) holds if and
only if Eq.~(\ref{eq:defi:admissible Novikov algebras1}) holds,
that is, in this case, an anti-pre-Lie algebra satisfying
Eq.~(\ref{anti-pre-Lie-pre}) is exactly an admissible Novikov
algebra.
\end{rmk}


\begin{thm}\label{main-2}
Let $A$ be a vector space with a bilinear operation $\ast$.  Then
$(A,\ast)$ is a pre-Lie algebra satisfying
Eq.~\eqref{pre-Lie-anti} if and only if its $(-q)$-algebra
$(A,\circ)$ is an anti-pre-Lie algebra satisfying
Eq.~\eqref{anti-pre-Lie-pre}.
\end{thm}
\begin{proof}
It is similar to the one for Theorem \ref{main-1}.
\end{proof}

In particular, when $q=-2$, the following conclusion has already
been given in \cite{LB}.

\begin{cor} Let $A$ be a vector space with a bilinear operation $\ast$.  Then
$(A,\ast)$ is a Novikov algebra if and only if its $2$-algebra
$(A,\circ)$ is an admissible Novikov algebra.
\end{cor}

\begin{cor}\label{den+Eq=pre+Eq} 
\begin{enumerate}
\item\label{it:11} Suppose that $(A,\succ,\prec)$ is a dendriform
algebra satisfying Eqs.~\eqref{eq:s1}-\eqref{dendri-anti-dendri}.
Then its associated pre-Lie algebra $(A, \ast)$ defined by
Eq.~(\ref{eq:pre}) satisfies Eq.~\eqref{pre-Lie-anti}.
 In particular, when $q=-2$, the associated pre-Lie
algebra of a Novikov-type dendriform algebra is a Novikov algebra.
\item \label{it:12} Suppose that
$(A,\vartriangleright,\vartriangleleft)$ is an anti-dendriform
algebra satisfying
Eqs.~\eqref{eq:anti-s1}-\eqref{anti-dendri-dendri}. Then its
associated anti-pre-Lie algebra $(A,\circ)$ satisfies
Eq.~\eqref{anti-pre-Lie-pre}.
 In particular, when $q=-2$, the associated anti-pre-Lie
algebra of an admissible Novikov-type dendriform algebra is an
admissible Novikov algebra.
\end{enumerate}
\end{cor}
\begin{proof}
(\ref{it:11}). Note that the $q$-algebra of $(A,\succ,\prec)$ is
an anti-dendriform algebra
$(A,\vartriangleright,\vartriangleleft)$ by Proposition
\ref{dendriform-anti-dendriform}. Let $(A,\circ)$ be the
associated anti-pre-Lie algebra of
$(A,\vartriangleright,\vartriangleleft)$. Then we have
$$x\circ y=x\vartriangleright y-y\vartriangleleft x=x\succ y
+qx\prec y -y\prec x-qy\succ x=x\ast y-q y\ast x,\;\;\forall
x,y\in A,$$ that is, $(A,\circ)$ is the $-q$-algebra of $(A,
\ast)$. By Proposition \ref{pre-Lie-anti-pre-Lie},  $(A, \ast)$
satisfies Eq.~(\ref{pre-Lie-anti}). The conclusion for the special
case that $q=-2$ follows straightforwardly.

 (\ref{it:12}). It is similar to the proof of Item~(\ref{it:11}).
\end{proof}




Combining Theorems \ref{main-1}, \ref{main-2} and Corollary
\ref{den+Eq=pre+Eq} together, we have the following commutative
diagram which is consistent with both the
diagrams~(\ref{eq:dendri}) and~(\ref{diag:comm}).
 \begin{equation} \label{eq:rbdiag-1}
\begin{split}
{\tiny  \xymatrix{
     \txt{dendriform algebra \\ $(A,\succ,\prec)$}
    \ar@{->}[r]^-{x\ast y=x\succ y-y\prec x}
    & \txt{pre-Lie algebra\\$(A,\ast)$} \\
   \txt{dendriform algebra  \\$(A,\succ,\prec)$+Eqs.~(\ref{eq:s1})-(\ref{dendri-anti-dendri})}
    \ar@{_{(}->}[u] \ar@<1mm>@{->}[d]^-{q-\text{algebra}}
    \ar@{->}[r]^-{x\ast y=x\succ y-y\prec x} & \txt{\ \ \ pre-Lie algebra \\  $(A,\ast)$+Eq. \eqref{pre-Lie-anti}} \ar@{^{(}->}[u] \ar@<1mm>@{->}[d]^-{-q-\text{algebra}}\\
     \txt{anti-dendriform algebra\ \ \ \ \ \ \ \ \\ $(A,\vartriangleright,\vartriangleleft)$+Eqs.~(\ref{eq:anti-s1})-(\ref{anti-dendri-dendri})}
     \ar@<1mm>@{->}[u]^-{-q-\text{algebra}}
    \ar@{->}[r]^-{x\circ y=x\vartriangleright y-y\vartriangleleft x}
      \ar@{^{(}->}[d]^-{} & \txt{\ \ \ \ \ \ \ \ anti-pre-Lie algebra\\ $(A,\circ)$+Eq. \eqref{anti-pre-Lie-pre}} \ar@<1mm>@{->}[u]^-{q-\text{algebra}} \ar@{_{(}->}[d]^-{}\\
      \txt{anti-dendriform algebra \\ $(A,\vartriangleright,\vartriangleleft)$}
    \ar@{->}[r]^-{x\circ y=x\vartriangleright y-y\vartriangleleft x}& \txt{anti-pre-Lie algebra\\ $(A,\circ)$}
}}
\end{split}
\end{equation}


In particular, when $q=-2,$ we have the following commutative
diagram:
 \begin{equation} \label{eq:rbdiag-2}
\begin{split}
{\tiny  \xymatrix{
     \txt{dendriform algebra \\ $(A,\succ,\prec)$}
    \ar@{->}[r]^-{x\ast y=x\succ y-y\prec x}
    & \txt{pre-Lie algebra\\$(A,\ast)$} \\
   \txt{Novikov-type dendriform algebra \\  $(A,\succ,\prec)$}
    \ar@{_{(}->}[u] \ar@<1mm>@{->}[d]^-{-2-\text{algebra}}
    \ar@{->}[r]^-{x\ast y=x\succ y-y\prec x} & \txt{Novikov algebra \\  $(A,\ast)$} \ar@{^{(}->}[u] \ar@<1mm>@{->}[d]^-{2-\text{algebra}}\\
     \txt{admissible Novikov-type dendriform algebra \ \ \ \ \ \ \\ $(A,\vartriangleright,\vartriangleleft)$} \ar@<1mm>@{->}[u]^-{2-\text{algebra}}
    \ar@{->}[r]^-{x\circ y=x\vartriangleright y-y\vartriangleleft x}
      \ar@{^{(}->}[d]^-{} & \txt{\ \ \ \ \ \ admissible Novikov algebra\\ $(A,\circ)$} \ar@<1mm>@{->}[u]^-{-2-\text{algebra}} \ar@{_{(}->}[d]^-{}\\
      \txt{anti-dendriform algebra \\ $(A,\vartriangleright,\vartriangleleft)$}
    \ar@{->}[r]^-{x\circ y=x\vartriangleright y-y\vartriangleleft x}& \txt{anti-pre-Lie algebra\\ $(A,\circ)$}
}}
\end{split}
\end{equation}

The above commutative diagram illustrates that it is reasonable to
regard Novikov-type dendriform and admissible Novikov-type
dendriform algebras  as ``analogues" of Novikov and
admissible Novikov algebras for dendriform and anti-dendriform
algebras respectively, justifying the notions of the former.

\section{General framework: analogues of anti-dendriform algebras and a new splitting of operations }

Illustrated by the study of anti-dendriform algebras in the
previous sections, we provide a general framework of introducing
the notions of analogues of anti-dendriform algebras to interpret
a new approach of splitting operations. We also characterize such
a construction in terms of double spaces.


We commence to use associative algebras as an example to exhibit the new
approach of splitting  operations, which is interpreted by a
general framework of introducing the notions of analogues of
anti-dendriform algebras. At first, we consider ``splitting the associativity", that is, expressing
the multiplication of an associative algebra as the sum of a
string of bilinear operations. Explicitly, let $(A,\cdot)$ be an
associative algebra and $(\cdot_i)_{1\leq i\leq N}:A\otimes
A\rightarrow A$ be a family of bilinear operations on $A$. Then
the operation $\cdot$ splits into the $N$ operations
$\cdot_1,\cdots, \cdot_N$ if
\begin{equation} x\cdot y=\sum_{i=1}^N x\cdot_iy,\;\;\forall x,y\in A.\end{equation}

\begin{ex} The ordinary operations splitting the associativity give
the following so-called {\bf Loday algebras}.
\begin{enumerate}
\item $N=2$: dendriform algebra (\cite{Lo}); \item $N=3$:
tridendriform algebra (\cite{LR}); \item $N=4$: quadri-algebra
(\cite{AL}); \item $N=8$: octo-algebra (\cite{Le3}); \item $N=9$:
ennea-algebra (\cite{Le1}).
\end{enumerate}
\end{ex}

For the case that $N=2^n$, $n=0,1,2,\cdots$, there is the
following ``rule" of constructing Loday algebras: by induction,
for the algebra $(A, \cdot_i)_{1\leq i\leq 2^n}$, besides the
natural (regular) representation of $A$ on
 the underlying vector space of
$A$ itself given by the left and right multiplication operators,
one can introduce the $2^{n+1}$ operations $\{\cdot_{i_1},
\cdot_{i_2}\}_{1\leq i\leq 2^n}$ such that
\begin{equation}x\cdot_i y=x\cdot_{i_1} y+x\cdot_{i_2}y,\;\;\forall x,y\in A, \; 1\leq i\leq
2^n,\end{equation} and their left and right multiplication
operators give a representation of $(A, \cdot_i)_{1\leq i\leq
2^n}$ by acting on the underlying vector space of $A$ itself. In
the sense of \cite{BBGN}, these Loday algebras are the successors'
algebras starting from the associative algebras.

Now we consider to construct analogues of anti-dendriform algebras
by the following ``rule" as another approach of splitting the
associativity. Let $N=2^n$, $n=0,1,2,\cdots$. By induction, for
the algebra $(A, \cdot_i)_{1\leq i\leq 2^n}$, 
one can introduce the $2^{n+1}$ operations $\{\cdot_{i_1},
\cdot_{i_2}\}_{1\leq i\leq 2^n}$ such that
\begin{equation}x\cdot_i y=x\cdot_{i_1} y+x\cdot_{i_2}y,\;\;\forall x,y\in A, \; 1\leq i\leq
2^n,\end{equation} and their negative left and right
multiplication operators give a representation of $(A,
\cdot_i)_{1\leq i\leq 2^n}$ by acting on the underlying vector
space of $A$ itself. Hence these algebras can be regarded as the
``anti-structures" for  the successors' algebras starting from the
associative algebras.




\begin{ex}
When $N=2$, that is, $n=1$, the corresponding algebra
$(A,\cdot_i)_{1\leq i\leq 2}=(A,\cdot_1,\cdot_2)$ is an
anti-dendriform algebra.
\end{ex}

Similarly, we consider the following approach of splitting the Lie
bracket of a Lie algebra in which anti-pre-Lie algebras are
included.

Let $(X,[\ ,\ ])$ be a Lie algebra and $(\cdot_i)_{1\leq i\leq
N}:X\otimes X\rightarrow X$ be a family of bilinear operations on
$X$. Then the Lie bracket $[\ ,\ ]$ splits into the commutator of
$N$ bilinear operations $\cdot_1,\cdots, \cdot_N$ if
\begin{equation}
[x,y]=\sum_{i=1}^N (x\cdot_iy-y\cdot_ix),\;\;\forall x,y\in
X.\label{eq:commu}
\end{equation}


For the case that $N=2^n$, $n=0,1,2,\cdots$, there is a ``rule''
of constructing the bilinear operations $\cdot_i$ as follows. 
By induction, for the algebra $(X,
\cdot_i)_{1\leq i\leq 2^n}$, 
one can introduce the $2^{n+1}$ bilinear operations
$\{\cdot_{i_1}, \cdot_{i_2}\}_{1\leq i\leq 2^n}$ such that
\begin{equation}
x\cdot_i y=x\cdot_{i_1} y-y\cdot_{i_2}x,\;\;\forall x,y\in A, \;
1\leq i\leq 2^n, \label{eq:rule}\end{equation} and their negative
left or right multiplication operators give a representation of
$(X, \cdot_i)_{1\leq i\leq 2^n}$ by acting on the underlying
vector space of $A$ itself. These algebras can be regarded as the
``anti-structures" for  the successors' algebras starting from the
Lie algebras.

\begin{ex}
 When $N=1$, that is, $n=0$,  the corresponding algebra $(X,\cdot_1)$
is an anti-pre-Lie algebra.
\end{ex}

In a summary, such ``anti-structures" as the ``counterparts" of
the  successors' algebras, which are put into the above general
framework as analogues of anti-dendriform algebras as well as
anti-pre-Lie algebras, provide a new splitting of operations. The
study on these structures such as the relationships with
anti-$\mathcal O$-operators and anti-Rota-Baxter operators,
 the correspondences between the subclasses of successors'
algebras and their anti-counterparts in terms of $q$-algebras, and
the operadic interpretation is expected in the future works.

At the end of this section, we give the following characterization
of these ``anti-structures" 
in terms of double spaces, motivated by
Corollary~\ref{cor:anti-d}.

Let $\mathcal C$ denote the category of all algebras $(A,\cdot)$
which satisfy a given set of multilinear relations $\mathcal
R_1=0$, $\cdots$, $\mathcal R_k=0$.

\begin{defi}
An algebra $(A,\vartriangleright,\vartriangleleft)$ is called a
{\bf $\mathcal C$-anti-dendriform algebra} if $(A\oplus A,
\cdot)\in\mathcal C$, where $\cdot$ is defined by
Eq.~(\ref{eq:asso-double}).
\end{defi}

Similarly, one can characterize the anti-structures for the
algebras $(A,\cdot_i)_{1\leq i\leq N}$ with $N=2^n$,
$n=0,1,2,\cdots$ as follows. By induction, for the algebra $(A,
\cdot_i)_{1\leq i\leq 2^n}$ giving the category $\mathcal
C_{2^n}$, one can introduce the $2^{n+1}$ operations
$\{\cdot_{i_1}, \cdot_{i_2}\}_{1\leq i\leq 2^n}$ such that
$(A\oplus A, \cdot_1,\cdots,\cdot_{2^n})\in \mathcal C_{2^n}$,
where $\cdot_i$ ($1\leq i\leq 2^n$) is defined by

\begin{equation}
\label{eq:asso-double11} (x,a)\cdot_i (y,b)=(x\cdot_{i1}
y+x\cdot_{i2} y, -x\cdot_{i1} b-a\cdot_{i2} y),\;\;\forall
x,y,a,b\in A.
\end{equation}

 \bigskip


\bigskip

{\bf Acknowledgments } This work is partially supported by NSFC
(11931009, 12271265), the Fundamental Research Funds for the
Central Universities and Nankai Zhide Foundation.

\end{document}